\documentclass[12pt]{amsart}

\usepackage{bbm, amsmath, amsfonts, amscd, latexsym, amsthm, amssymb, graphicx,mathrsfs, mathabx, comment}
\usepackage[hang,flushmargin]{footmisc}
\usepackage[colorlinks=true]{hyperref}
\usepackage{abstract}
\usepackage{mathtools}
\usepackage{caption}
\usepackage{tikz}
\usepackage{float}

\usetikzlibrary{patterns}
\usepackage{color, soul}

\usepackage{tikz-cd}
\tikzset{
	symbol/.style={
		draw=none,
		every to/.append style={
			edge node={node [sloped, allow upside down, auto=false]{$#1$}}}
	}
}

\makeatletter
\newif\if@check@engine  \@check@enginetrue 
\makeatother
\usepackage[usenames,dvipsnames]{pstricks}

\newtheorem{theor}{\hspace{1cm}{\sc Theorem}}[section]
\newtheorem{utver}[theor]{\hspace{1cm}{\sc Proposition}}
\newtheorem{predpol}[theor]{\hspace{1cm}{\sc Assumption}}
\newtheorem{sledst}[theor]{\hspace{1cm}{\sc Corollary}}
\newtheorem{lemma}[theor]{\hspace{1cm}{\sc Lemma}}
\newtheorem{conj}[theor]{\hspace{1cm}{\sc Conjecture}}
\newtheorem*{utver*}{\hspace{1cm}{\sc Proposition}}
\theoremstyle{definition}

\newtheorem{defin}[theor]{\hspace{1cm}{\sc Definition}}
\newtheorem{exa}[theor]{\hspace{1cm}{\sc Example}}
\newtheorem{observ}[theor]{\hspace{1cm}{\sc Observation}}
\newtheorem{rem}[theor]{\hspace{1cm}{\sc Remark}}

\newtheorem{question}[theor]{\hspace{1cm}{\sc Question}}

\newcommand{\ind}{\mathop{\rm ind}\nolimits}

\newcommand{\codim}{\mathop{\rm codim}\nolimits}

\newcommand{\sing}{\mathop{\rm sing}\nolimits}

\newcommand{\rk}{\mathop{\rm rk}\nolimits}
\newcommand{\crk}{\mathop{\rm crk}\nolimits}

\newcommand{\vol}{\mathop{\rm Vol}\nolimits}
\newcommand{\conv}{\mathop{\rm conv}\nolimits}

\newcommand{\Trop}{\mathop{\rm Trop}\nolimits}

\newcommand{\MV}{\mathop{\rm MV}\nolimits}

\newcommand{\MP}{\mathop{\rm MP}\nolimits}

\newcommand{\supp}{\mathop{\rm supp}\nolimits}

\renewcommand{\emph}[1]{{\it {\color{NavyBlue} #1}}}
\newcommand{\bigslant}[2]{{\raisebox{.2em}{$#1$}\left/\raisebox{-.2em}{$#2$}\right.}}

\def\R{\mathbb R}
\def\N{\mathbb N}
\def\Z{\mathbb Z}
\def\Q{\mathbb Q}
\def\C{\mathbb C}
\def\CC{({\mathbb C}^\star)}

\def\CP{\mathbb C\mathbb P}

\def\x{\times}
\newcommand{\ord}{\mathop{\rm ord}\nolimits}
\newcommand{\cork}{\mathop{\rm cork}\nolimits}
\begin{document}

\begin{center}
\title{Sparse curve singularities, singular loci of resultants, and Vandermonde matrices}
{\textsc{\Large
Sparse curve singularities, singular loci of resultants, and Vandermonde matrices
}}\\[3ex]

Alexander Esterov$^{1}$,
Evgeny Statnik$^{1}$,
Arina Voorhaar$^{2}$.\\[1ex]

{\footnotesize $^{1}$HSE University, Moscow, Russia\\
$^{2}$ Department of Mathematical Sciences, University of Copenhagen, Denmark}
\end{center}
{\let\thefootnote\relax\footnotetext{\noindent {\bf E-mail:} aesterov@hse.ru,
evstatnik@gmail.com, arina.voorhaar@gmail.com
\newline {\bf 2020 Mathematics Subject Classification:} 32S25, 14N10, 14M25
}}

\begin{abstract} We compute the $\delta$-invariant of a curve singularity parameterized by generic sparse polynomials. We apply this to describe topological types of generic singularities of sparse resultants and ``algebraic knot diagrams'' (i.e. generic algebraic spatial curve projections).

Our approach is based on some new results on zero loci of Schur polynomials, on transversality properties of maps defined by sparse polynomials, and on a new refinement of the notion of tropicalization of a curve (ultratropicalization), which may be of independent interest.
\end{abstract}

\tableofcontents

\section{Introduction} \label{Sintro}

In this paper, we discuss complex algebraic plane curve singularities up to their topological equivalence.

\subsection{$\delta$-invariants of sparse curve singularities} A sparse analytic function singularity  $\C^k\to\C$ is the singularity at 0 of a generic linear combination of a given set of monomials (called the support set). The motivation to study sparse singularities comes from the observation that many classes of ``the most interesting'' singularities (such as simple singularities) are sparse. Important invariants of sparse singularities were computed in terms of the support set (or its Newton polyhedron) by many authors, starting from Kouchnirenko's formula for the Milnor number \cite{kouchn}.

In the same way, it is natural to consider parameterized sparse singularities, e.g. define a sparse curve singularity as the singularity of a map $f:\C\to\C^k$, whose components are generic linear combinations of prescribed sets of monomials $B_1,\ldots,B_k\subset\Z^1_{\geqslant 0}$ (encoding a monomial by its degree). Again, most of classification problems for singularities of curves (as well as some other questions, see below) lead to lists of sparse singularities: see e.g. \cite{classif1}, \cite{classif2}, \cite{classif2a}, \cite{classif3}, \cite{classif3a}, \cite{classif3b}, \cite{classif4}, \cite{classif5}.

However, even the basic invariants of sparse curve singularities are not computed in terms of the support sets $B_1,\ldots,B_k$ so far. We do this for $k=2$ and conjecture the answer in general.

Let $d_i:=\min B_i>0$, and 
$j_r:={\rm GCD}\,\bigcup_i(B_i\cap[d_i,d_i+r])$. We assume that $j_\infty=1$ (or, equivalently, $f$ is an injective map germ).

\begin{theor}\label{thsparse}
1) The $\delta$-invariant of the curve singularity $$f:(\C,0)\to(\C^2,0)$$ whose components $(f_1,f_2)=f$ are generic analytic functions supported at $B_1$ and $B_2\subset\Z^1_{\geqslant 0}$, equals $$\delta_{B_1,B_2}:=\dfrac{(d_1-1)(d_2-1)}2+\sum\limits_{r=0}^{\infty}\dfrac{j_r-1}{2},$$ and thus the Milnor number equals $2\delta_{B_1,B_2}$.

2) More specifically, $f$ has the expected $\delta$-invariant if and only if it is 0-nondegenerate (see Definition \ref{def0nondeg}), and a strictly higher invariant otherwise.

3) Moreover, all 0-nondegenerate $f$ supported at $(B_1,B_2)$ are topologically equivalent to each other. This topological type will be called the $(B_1,B_2)$-sparse singularity. 
\end{theor}

\begin{conj}\label{conj11}
We expect that, for arbitrary $k$, the $\delta$-invariant of the sparse curve singularity in $\C^k$, whose components are supported at $B_1,\ldots,B_k$, equals $$(\mbox{a function of }d_1,\ldots,d_k)+\sum\limits_{r=0}^{\infty}\dfrac{j_r-1}{2}.$$
\end{conj}

\begin{question}
1. How to construct a resolution of a sparse singularity?

2. How to compute its further invariants, such as the monodromy operator and the Puiseux characteristics?

3. Does there exist canonical inclusion-minimal subsets $B'_i\subset B_i$, such that the $B_\bullet$-sparse singularity is topologically equivalent to the $B'_\bullet$ one?
\end{question}

\begin{rem} 1. If $B_1;=\{d_1\}$, then the Puiseux characteristics of the $(B_1,B_2)$-sparse singularity are the numbers $d_2+r$ such that $j_{r-1}>j_r$, so in this special case the aforementioned questions are well studied (see \cite{ao} and \cite{o09}). In particular, our Theorem \ref{thsparse} can be regarded as a generalization of Thorem 5.1 in \cite{ao}, though our proof is completely different.

2. Our proof of the theorem for $k=2$ is based on the topological interpretation of the Milnor number (see Section \ref{sssc}), and is thus not applicable to the case $k>2$ (which would require an algebraic approach in the vein of \cite{greuel}).

3. If Conjecture \ref{conj11} is valid, then the function of $d_i$'s therein is not elementary: e.g. in the case of a monomial curve (i.e. $B_i=\{d_i\}$), it is the genus of the semigroup generated by $d_i$'s.

4. If the support sets $B_i$ are finite, then $f$ is a polynomial map, and its image may have other non-trivial singularities (besides the sparse one). We describe them in Example \ref{exa0}.

\end{rem}

We are particularly interested in sparse curve singularities because they appear as singularities of sparse resultants and "algebraic knot diagrams" (algebraic spatial curve projections). 

Theorem \ref{thsparse} allows to completely describe generic singularities of these objects, see Theorems \ref{thres} and \ref{thproj} below. We work out both of them in one paper because their proofs are mutually dependent.

\subsection{Singular loci of resultants}
Given a pair of finite support sets $B=(B_1,B_2)$ in $\Z$, the resultant $R_B$ is the closure of the set of all pairs of polynomials supported at $(B_1,B_2)$ and having a common root in $\C^*$. Singular loci of sparse resultants and discriminants appear in various topics such as enumerative geometry (see e.g. an overview in \cite{E18}) or symbolic algebra (see e.g. \cite{dickenst21}). Degrees and, more generally, tropicalizations of such singular loci are computed (see e.g. \cite{E18} and \cite{dickenst16}) under the assumption that they ``look as expected'' i.e. as in the classical case.

For the classical resultant of polynomials of given degrees $d_i$ (i.e. for $B_i=[0,d_i]$), the singular locus $\sing R_B$ is an irreducible algebraic set of codimension 2, and the transversal singularity type of the resultant at the generic $f\in\sing R_B$ is an ordinary double point. (The latter statement means that a generic plane through the point $f$ intersects $R_B$ by a curve that locally splits into two smooth and transversally intersecting branches.)

The classical proof involves certain (obvious) properties of the Vandermonde determinant. As soon as we generalize these properties to more general Schur polynomials in Section \ref{sschur}, we obtain a similar description of $\sing R_B$ for arbitrary $B_1$ and $B_2$.
\begin{theor}\label{thres}
Exclude the following exceptional cases from our consideration:

-- one of $B_i$'s consists of one element (the resultant is trivial);

-- $B_i=\{p_i,p_i+q\}$ (the resultant is the determinantal hypersurface $\{c_1c_2=c_3c_4\}\subset\C^4$).

Then the set $\sing R_B$ has only codimension 2 irreducible components, splitting into the following groups:

--  for every $m\geqslant 3$, at most two components at whose generic point the transversal singularity type of the resultant is the ordinary $m$-point (i.e. the union of $m$ smooth pairwise transversal curves);

-- several components at whose generic point the transversal singularity type is the ordinary double point;

-- at most two components at whose generic point the transversal singularity type is a sparse curve singularity.
\end{theor}
A full version of this theorem (with an explicit description of the strata and their degrees) is given in Section \ref{s3} (see Theorem \ref{th0res}). In the course of its proof, we characterize the singular points of the resultant that satisfy the Whitney condition (b), see Proposition \ref{s0whitn}.

Speaking more specifically of the results on Schur polynomials that we need here, we consider a generalized Vandermonde matrix $M(x_0,\ldots,x_k)=(x_i^{b_j})$ for a given tuple of integers $(b_1,\ldots,b_m)$, and study its degeneracy locus (i.e. the set of $x\in\C^k$ such that $\rk M(x)<\max$). See e.g. \cite{zannier} and \cite{bshapiro} for some fundamental results and conjectures on the geometry of such sets. In a subsequent Lemma \ref{lemMsplits2}, we completely characterize (over a field of zero characteristic) tuples of roots of unity $x=(x_0,x_1,x_2)$ for which $\rk M(x)<\max$, proving the following conjecture for $k\leqslant 2$.
\begin{conj}
Let $X=\{ x_0=1, x_1 \ldots, x_k \}$ be a set of pairwise different roots of unity, let $B=\{ b_0=0, b_1, \ldots, b_m \}$ be a set of coprime integer numbers,  and assume that the matrix $M=(x_i^{b_j})$ is degenerate. Then $B$ can be split into $B_1\sqcup\cdots\sqcup B_k$, such that all $x_i$'s are the roots of unity of degree $\gcd\{b-b'\,|\,b,b'\in B_i,\,i=1,\ldots,k\}$.
\end{conj}
Note that Lemma \ref{lem3prop} reduces the general conjecture to the case of a square matrix, i.e. essentially to a statement about the corresponding Schur polynomial (which is proved for $k=2$ in Lemma \ref{lem3minor}). For arbitrary $k$, this conjecture is an important step in towards extending Theorem \ref{thres} to higher codimension strata of the resultant.

\subsection{Singularities of algebraic knot diagrams} As in the univariate case, we identify monomials $x_1^{a_1}\cdots x_k^{a_k}$ with their degrees $(a_1,\ldots,a_k)$ as points of the lattice $\Z^k$, and define a sparse polynomial supported at $A\subset\Z^k$ as a generic linear combination of monomials from $A$.

Consider a spatial curve in the algebraic torus $\CC^3$ defined as the common zero locus of a pair of sparse polynomials supported at given sets $A_1$ and $A_2\subset\Z^3$.
We are interested in the singularities of its image under the coordinate projection $\CC^3\to\CC^2,\,(x_1,x_2,x_3)\mapsto(x_1,x_2)$. Since the image may be not closed, we study its closure $C\subset\CC^2$ (recall that the metric and Zariski closures are the same thing for the image of an algebraic set).

\begin{theor} \label{thproj} All singularities of the sparse curve projection $C$ are ordinary multiple points and sparse curve singularities.
\end{theor}
The full version of this theorem, specifying the exact number of singularities of each type except for ordinary double points, is stated and proved in Section \ref{s3} (see Theorem \ref{th0proj}). For the number of ordinary double points, see the next Subsection \ref{ssutrop}.

Moreover, from the full version of Theorem \ref{thproj} we will see that all ordinary $m$-points for $m\geqslant 3$ come from $m$-tuples of points of the form $(x_1,x_2,\sqrt[m]{1}^kx_3),\,k=1,\ldots,m,$ on the spatial curve, unless 

i) one of the $A_i$'s is contained in a union of two horizontal planes (by horizontal we mean a plane parallel to the first two coordinate axes), or 

ii) there exists a union of three horizontal planes containing both $A_1$ and $A_2$ up to a shift. 

Since at most two of the $m$ numbers $\sqrt[m]{1}^k$ are real, this implies the following.
\begin{sledst}\label{c17}
Consider a spatial curve $V\subset(\R\setminus 0)^3$ defined by generic real polynomials supported at $A_1$ and $A_2\subset\Z^3$ that do not satisfy (i) and (ii) above. Then the closure $C\subset(\R\setminus 0)^2$ of the coordinate projection of $V$ is topologically a link diagram, in the sense that

1) $C$ is homeomorphic to $\R$ near every its point except for finitely many;

2) every exceptional point of $C$ is a real ordinary double point.
\end{sledst}
Such objects may be interesting from the perspective of real algebraic knots (see e.g. \cite{bjo}, \cite{mikhorevk}, \cite{mikhorevk1}), because a natural way to construct examples of real algebraic links with prescribed topology is the patchworking technique for spatial sparse curves as in \cite{sturmfci} and \cite{bih}. Such constructions produce real spatial complete intersection curves defined by sparse polynomial equations, and are thus covered by Corollary \ref{c17}.

\subsection{Ultratropicalizations.}\label{ssutrop} In the setting of Theorem \ref{thproj},
the number of ordinary double points of the projection can be found using the subsequent Theorem \ref{thsum}, which computes the sum of the $\delta$-invariants for all singularities of the sparse curve projection $C$. 
\begin{observ}\label{obnumdouble} We have:

\quad $($the sum of the $\delta$-invariants of all singularities, known from Theorem \ref{thsum}$)=$

$=($the sum of the $\delta$-invariants of the 
ordinary double points, 

\hfill equal to the sought number of such points$)+$

\;$+($the sum of the $\delta$-invariants of the other singularities, 

\hfill known from Theorem \ref{th0proj},  the full version of Theorem \ref{thproj}$)$.

\end{observ}

While we do not formulate Theorem \ref{thsum} in the introduction due to involved notation, it is a key result in our story. 

Instead, we outline here a more general formula for the sum of the $\delta$-invariants of an arbitrary reduced curve $C\subset\CC^2$, which will eventually lead to Theorem \ref{thsum}.

An important invariant of an algebraic curve $C\subset\CC^k$ is its tropical fan $\Trop C$, a finite collection of rays $R\in\mathcal{R}$ in $\Z^k$ emanating from 0 with integer multiplicities $m_R$.

In Section \ref{s2}, we assign to $C$ a refined invariant called the ultratropical fan ${\rm UTrop}\, C$, which consists of the same rays $R_i$, equipped with symmetric integer matrices $(g_{i,j}^{R})$ of size $m_R\times m_R$. The latter is called {\it the tangency matrix} in the direction $R$ and consists of the orders of contact of the Puiseux roots of the curve $C$ with respect to the $R_i$-orbit of the tropical compactification of $C$ (see Definition \ref{deftangency} for details).
\begin{rem}
The name ``ultratropical'' reflects the fact that the inverse numbers to the entries $g_{i,j}^{R}$ define an ultrametric on the set of Puiseux roots of the curve $C$ in the direction $R$, see Section \ref{s2} for details.
\end{rem}

Once we know the ultratropicalization ${\rm UTrop}\, C$, we can compute the sum of the $\delta$-invariants of a plane curve $C$.
\begin{lemma}\label{labstr}
1. The sum of the Milnor numbers of the singularities of a reduced curve $C\subset\CC^2$ equals $$e(C)+(\Trop C)^2-\sum_{R, i, j}g^R_{i,j},$$
where the three terms are the Euler characteristic, the self-intersection index of the tropical fan (equal to the lattice area of the Newton polygon of $C$), and the sum of all entries of all the tangency matrices in ${\rm UTrop}\, C$.

2. Equivalently, the sum of the $\delta$-invariants of a reduced curve $C\subset\CC^2$ equals one half of
$$e(n(C))+(\Trop C)^2-\sum_{R, i, j}g^R_{i,j},$$
where $n(C)$ is the normalization of $C$.
\end{lemma}
This essentially follows from the Euler characteristics count, see Section \ref{s2}. 
\begin{rem}
Generalizing this lemma (and the notion of ultratropicalization) to higher dimensional tori and their higher dimensional subvarieties would allow to extend many results in this paper to higher dimension.
\end{rem}
It is a classical fact tracing back to \cite{maurer80} that the tropicalization of the image $C$ of a curve $V$ under the coordinate projection $\CC^n\to\CC^k$ is expressed in terms of $\Trop V$. For ultratropicalizations this is not true: the tangency matrices of ${\rm UTrop}\, V$ are just a lower bound for those of ${\rm UTrop}\, C$. However, 
Proposition \ref{projtangency} of Section \ref{ssutropproj} completely describes ${\rm UTrop}\, C$, if $V$ is defined by sparse polynomial equations. For $n=3$ and $k=2$, together with the preceding lemma and the isomorphism $V=n(C)$, this yields Theorem \ref{thsum}.

\subsection{Structure of the paper} In Section \ref{s3}, we fully formulate Theorems \ref{thres} and \ref{thproj}, the full versions of the theorems in the introduction. 

The other four sections introduce the techniques intended for the proof of the main results. They have a singular, tropical, arithmetic, and singular flavor respectively, and can be read independently of each other.

In Section \ref{sssc}, we prove Theorem \ref{thsparse} and explicitly formulate its genericity assumptions.

In Section \ref{s2} we introduce ultratropicalizations and prove 

the facts from the preceding Section \ref{ssutrop}.

In Section \ref{sschur}, we study certain degenerations of Vandermonde matrices in order to enumerate ordinary  multiple singular points of resultants.

Finally, in Section \ref{ssproofs}, we use the techniques from singularity theory (Whitney stratifications and the Thom-Mather theorem) to finish the proof of the main results. 

A key step in this proof is the observation that a map $\CC^m\to\C^n$, whose components are generic sparse polynomials, is transversal to any given stratified algebraic set in $\C^n$. This is obvious in $\CC^n$ and demands slightly greater care at the coordinate hyperplanes, see Theorem \ref{thtransv} and Corollary \ref{sltransv1} for precise statements.

\subsection*{Acknowledgements}
The third author was funded by Horizon Europe ERC (Grant number: 101045750, Project acronym: HodgeGeoComb).

\section{Singularities of sparse resultants and algebraic knot diagrams}\label{s3}

We identify monomials $x_1^{a_1}\cdots x_k^{a_k}$ with their degrees $(a_1,\ldots,a_k)$ as points of the lattice $\Z^k$, define a polynomial supported at a finite set $A\subset\Z^k$ as a linear combination of the monomials from $A$, and denote the vector space of such polynomials by $\C^A$. Similarly, for not necessarily finite $A\subset\Z^k_{\geqslant 0}$, we denote by $\C^A$ the space of power series supported at $A$ and convergent near the origin.

Recall that a sparse curve singularity supported at $B$ is defined to be the germ $$(f_1,\ldots,f_n)\,:\,(\C,0)\,\to\,(\C^n,f(0))\eqno{(*)}$$
whose components are generic sparse analytic functions $$f:=(f_1,\ldots,f_n)\in\C^{B_1}\oplus\cdots\oplus\C^{B_n}=:\C^B$$
supported at the given sets $B_1,\ldots,B_n\subset\Z,\,\min B_i\geqslant 0$.
\begin{rem}
1. Adding or removing $\{0\}$ to $B_i$ results in adding or removing a constant term from the polynomial $f_i$, which does not affect the singularity at $f(0)\in\C^n$.

2. If $d_B:={\rm GCD}(B_1\cup\ldots\cup B_n)>1$, then $f(t)=\tilde f(t^d)$ for $\tilde f_i\in\C^{B_i/d}$, i.e. every sparse curve singularity can be represented as a $B$-sparse curve singularity with $d_B=1$. We always assume this in what follows.
\end{rem}

\subsection{The components of the singular locus of the resultant}
For finite sets $B_i\subset\Z$ denote $$L(B_i):=\max B_i-\min B_i,$$ and define $$\langle B_1,\ldots,B_n\rangle$$ as the maximal number $m\in\Z$ such that each $B_i$ can be represented as $m\cdot B_i'+b_i$ for suitable $b_i$ and $B_i'\subset\Z$ (operations with finite subsets of $\Z$ are understood point-wise).
\begin{rem}\label{resgcd1}
Note that it is enough to study resultants $R_B\subset\C^{B_1}\oplus\C^{B_2}$ for $\langle B_1,B_2\rangle=1$, because the general case $\langle B_1,B_2\rangle=m>1$ can be reduced to this one by a change of variable $\tilde x=x^m$. More specifically, the natural isomorphism $\C^{B_1}\oplus\C^{B_2}\to\C^{B'_1}\oplus\C^{B'_2}$ sending to every $(f_1,f_2)$ the pair of polynomials $\bigl(x^{b_1}f_1(x^m),x^{b_2}f_2(x^m)\bigr)$, maps the initial resultant $R_{B_1,B_2}$ to $R_{B'_1,B'_2}$ with $\langle B'_1,B'_2\rangle=1$. 
\end{rem}
\begin{theor}\label{th0res} Exclude with no loss in generality the following cases:

-- one of $B_i$'s consists of one point (then the resultant is trivial), 

-- $\langle B_1,B_2\rangle>1$ (see Remark \ref{resgcd1});

-- $L(B_1)=L(B_2)=1$ (then the resultant is given by the equation $c_1c_2=c_3c_4$ in $\C^4$, and thus has one singular point).

\vspace{1ex}

Then all irreducible components of $\sing R_B$ have codimension 2 and the following transversal singularity types:

1) For every $m>1$, there exists at most one way to decompose one of $B_i's$ (say, $B_1$) into $B_{1,1}\sqcup B_{1,2}$ so that $\langle B_{1,1},B_{1,2},B_2\rangle=m$ and $|B_1|>2$. This decomposition produces the singular locus component
$$S_m:=\{(f_{1,1}+ f_{1,2},f_2)\,|\,\mbox{\rm the three polynomials } f_\bullet\in\C^{B_\bullet}\mbox{ \rm have a common root}\}.\eqno{(1)}$$
The transversal singular type of the resultant at the generic point of this component is an ordinary $m$-point.

2) Unless $(1+\min B_i)\in B_i$ for $i=1$ or 2, the set $$S_0=\{(f_1,f_2)\,|\, \mbox{ \rm the monomial } x^{\min B_i}\mbox{ \rm has coefficient }0\mbox{ \rm in }f_i \}\eqno{(2)}$$
is a singular locus component. Its transversal singularity type is that of the sparse curve singularity supported at $B_i-\min B_i,\, i=1,2$.

3) Similarly, unless $(\max B_i-1)\in B_i$ for $i=1$ or 2, the set $$S_\infty=\{(f_1,f_2)\,|\, \mbox{ \rm the monomial } x^{\max B_i}\mbox{ \rm has coefficient }0\mbox{ \rm in }f_i \}\eqno{(3)}$$
is a singularity locus component. Its transversal singularity type is that of the sparse curve singular supported at $(\max B_i)-B_i,\, i=1,2$.

4) If $|B_i|=2$ and $L(B_{3-i})>1$, then the set $$T_i=\{(f_1,f_2)\,|\, \mbox{ \rm both monomials of } f_i \mbox{ \rm have zero coefficients}\}\eqno{(4)}$$
is a singular locus component. Its transversal singularity type is an ordinary $L(B_{3-i})$-point.

5) If $B_1$ and $B_2$ differ by a shift $b\in\Z$ and consist of three points, then the set $$T_0=\{(f,g)\,|\,f/g=\lambda x^b\}$$ is a singular locus component. Its transversal singularity type is an ordinary $L(B_1)$-singularity.

6) If $B_1$ and $B_2$ are as in (5), then there are no more components in the singular locus. Otherwise the rest of the singular locus $S_1$ has the transversal singularity type of the ordinary double point and consists of the following components:

-- Denote $\langle B_i\rangle$ by $k_i$, and $\sqrt[k_i]{1}$ by $\epsilon_{k_i}$. For every $i=1,2$ such that $|B_{3-i}|>2$, and for every positive $j\leqslant k_i/2$ that is not divisible by an integer of the form $k_i/m$ for $m$ from (1), the closure $S_i^j$ of the set $$
\left\{ (f,g) \,|\, \exists x \in \C^\x: f(x)=g(x)=f\left(x\cdot\epsilon^j_{k_i}\right)=g\left(x\cdot \epsilon^j_{k_i}\right)=0\right\}$$
is a component of $S_1$. 

-- The closure $S$ of 

\vspace{1ex}

$\{(f_1,f_2)$ having two common roots $x_1,x_2\in\C^*$ such that $x_2/x_1$ is not a root of unity$\}$ 

\vspace{1ex}

\noindent is a component of $S_1$. 

\vspace{1ex}

The degrees of these strata and the $\delta$-invariants of their transversal singularity types are shown in the following table. 
\end{theor}
\noindent {\bf Table 1:} a stratum, its degree, the transversal singularity type, how it looks over $\R$, and its $\delta$-invariant
(recall that the quantity $\delta_{B_1,B_2}$ is defined in Theorem \ref{thsparse}).

\noindent \begin{tabular}{ |l|p{3.5cm}|p{3.2cm}|l|l| } \hline
$S_0$ & 1 & sparse singularity & $\prec$ & $\delta_0:=\delta_{B_1-\min B_1,B_2-\min B_2}$ \\ \hline
$S_\infty$ & 1 & sparse singularity & $\prec$ & $\delta_\infty:=\delta_{(\max B_1)-B_1,(\max B_2)-B_2}$   \\ \hline
$S_1$ & $\bigl((L(B_1+B_2)-1)^2$\newline ${}\quad -\delta_0-\delta_\infty+1\bigr)/2$ & ordinary \newline double point & $\times$ & 1  \\ \hline
$S_m,\, m>1$ & $L(B_{1,1}+B_{1,2}+B_2)$ & ordinary $m$-point & $\convolution$ & $m(m-1)/2$ \\ \hline
$T_i, i=1,2$ & 1 & ordinary $m$-point \newline for $m:=L(B_{3-i})$ & $\convolution$ & $m(m-1)/2$ \\ \hline
$T_0$ & 3 & ordinary $m$-point \newline for $m:=L(B_1)$ & $\convolution$ & $m(m-1)/2$ \\ \hline
\end{tabular} 

\vspace{1ex}

\subsection{Singularities of sparse curve projections}
For finite sets $A_1$ and $A_2\subset\Z^3$, a sparse spatial curve is the common zero locus of a pair of generic polynomials $(g_1,g_2)\in\C^{A_1}\oplus\C^{A_2}$.
We are interested in the singularities of its image under the coordinate projection $\CC^3\to\CC^2,\,(x,y_1,y_2)\mapsto(y_1,y_2)$. Since the image may be not closed, we study its closure $C\subset\CC^2$ (recall that the metric and Zariski closures are the same thing for the image of an algebraic set).

The answer will heavily depend on the images $B_i$ of the sets $A_i$ under the projection $\Z^3\to\Z^1,\, (a_1,a_2,a_3)\mapsto a_1$.
\begin{rem}\label{remsparseres}
Note that the resulting sets $B_i$ may be sparse (i.e. not entirely consisting of consecutive integers) even if $A_i$ are ``nice'' lattice polytopes. This is one of the reasons to study sparse resultants here.
\end{rem}
For every subset $B\subset B_i$ and $g_i(x)=\sum_{a\in A_i} c_ax^a$, we define $A_i^{B}\subset A_i$ as the preimage of $B$, and $g_i^{B}$ as $\sum_{a\in A_i^{B}} c_ax^a$. The lattice mixed volume of the convex hulls of $n$ sets $S_i\subset\Z^n$ (or, more generally, $k$ sets that can be shifted to the same $k$-dimensional coordinate plane in $\Z^n$) is denoted by $\MV(S_1,\ldots,S_k)\in\Z$.

\begin{theor}\label{th0proj}
Exclude from our consideration the same exceptional $B=(B_1,B_2)$ as in Theorem \ref{th0res}.
Then the curve $C\subset\CC^2$ has the following singularities. 

1) For every $m>1$, there exists at most one way to decompose one of $B_i's$ (say, $B_1$) into $B_{1,1}\sqcup B_{1,2}$ so that $\langle B_{1,1},B_{1,2},B_2\rangle=m$ and $|B_1|=1$. Then the solutions of the system $$g_1=g_2^{B_{2,1}}=g_2^{B_{2,2}}=0\eqno{(1)}$$ project to $\MV(A_1,A_{2}^{B_{2,1}},A_{2}^{B_{2,2}})/m$ ordinary $m$-points of $C$.

2) Additionally, if one of $B_i$'s (say, $B_1$) consists of two elements, then the solutions of the system $$g_1^{\min B_1}=g_1^{\max B_1}=0\eqno{(2)}$$ are $\MV(A_1^{\min B_1},A_1^{\max B_1})$ ordinary $L(B_2)$-points of $C$.

3) Unless $(1+\min B_i)\in B_i$ for $i=1$ or 2, the solutions to the system $$g_1^{\min B_1}=g_2^{\min B_2}=0\eqno{(3)}$$ are $\MV(A_1^{\min B_1},A_2^{\min B_2})$ sparse curve singularities of $C$, each supported at the sets $B_1-\min B_1$ and $B_2-\min B_2$.

4) The same for $\max$ instead of $\min$.

5) If $B_1=\{b_1,b_2,b_3\}$ and $B_2=\{b_1+b,b_2+b,b_3+b\}$, then the solutions of the system of three equations
$$g_1^{b_j}(1,x_2,x_3)=\lambda g_2^{b_j+a}(1,x_2,x_3),\,j=1,2,3,\mbox{ \rm in the three variables } \lambda,x_2,x_3\eqno{(4)}$$ project to the ordinary $L(B_1)$-points of $C$, whose number equals the mixed volume of $\bigl(A_1^{b_j}\times\{0\}\bigr)\cup \bigl(A_2^{b_j+b}\times\{1\}\bigr),\,j=1,2,3$.

6) Besides the listed singularities, the curve $C$ has only ordinary double points.
\end{theor}
\begin{rem}\label{remsum}
Recall that the number of ordinary double points in Part 6 can be found from the subsequent Theorem \ref{thsum}, see Observation \ref{obnumdouble}. 
\end{rem}
\begin{exa}\label{exa0}
Assume that $0\in B_i\subset\Z_{\geqslant 0},~i=1,2,$ and define the sets ${A_1=\{(b,0,0)\mid b\in B_1\}\cup\{(0,1,0)\}}$ and $A_2=\{(b,0,0)\mid b\in B_2\}\cup\{(0,0,1)\}$, as shown in Figure \ref{fig:MFP1}. Then, in the setting of Theorem \ref{th0proj}, the curve $C$ (whose Newton polygon is computed in Example \ref{exa:MFP}) may have ordinary $m$-points for various $m$ (whose number can be found using Theorem \ref{th0proj}), and one sparse singularity supported at $(B_1,B_2)$. 

Note that the latter fact can be verified directly from the definition of a sparse singularity (rather than via Theorem \ref{th0proj}): the equations defining the spatial curve scale to the form $f_1(x)-y_1=f_2(x)-y_2=0$ for $f_i\in\C^{B_i}$, i.e. $C$ is the sparse plane curve from the definition of the sparse singularity type. 
\end{exa}

\section{Proof of Theorem \ref{thsparse} on sparse curve singularities: 0-nondegeneracy}\label{sssc}

\subsection{Proof of Theorem \ref{thsparse}.1\&2.} 
\begin{lemma}\label{ldeltainters}
The $\delta$-invariant of an injective germ $(f_1,f_2):(\C,0)\to(\C^2,0)$ equals the intersection number $I$ of the curves $F_i(t_1,t_2)=0,\,i=1,2$, at $0\in\C^2$, where $$F_i(t_1,t_2):=\frac{f(t_2)-f(t_1)}{t_2-t_1}.\eqno{(*)}$$
\end{lemma}
\begin{proof}
Let $\tilde f_i$ be a generic small perturbation of a representative of the germ $f_i$, defined on some neighborhood of $0\in\C$, and $\tilde F_i$ the respective divided difference. Then the image of $(\tilde f_1,\tilde f_2)$ has only ordinary double self-intersection points $$\tilde f_\bullet(t_1)=\tilde f_\bullet(t_2),\,t_1\ne t_2.$$ On one hand, their number equals the $\delta$-invariant of the initial singularity (by one of its definitions), and on the other hand it tautologically equals the number of solutions of $$\tilde F_1(t_1,t_2)=F_2(t_1,t_2)=0,$$ i.e. the total intersection number of the curves $\tilde F_1(t_1,t_2)=0$ and $\tilde F_2(t_1,t_2)=0$ in $U\times U$. The latter equals the sought intersection number $I$ by the invariance of the total intersection number under deformations.
\end{proof}
This lemma reduces Theorem \ref{thsparse}.1 and 2 to the following.
\begin{utver}
Let $f_i$ be as in Theorem \ref{thsparse}, then the intersection number $I$ from Lemma \ref{ldeltainters} equals or exceeds the expected value $\delta_{B_1,B_2}$ once $(f_1,f_2)$ are 0-nondegenerate in the following sense.
\end{utver}

Let $B_1, B_2$ be sets in $\Z^1_{\geqslant 0}$ such that $\langle B_1,B_2\rangle=\Z,$ and $0\in B_1\cap B_2.$ Denote $d_i=\min(B_i\setminus\{0\}).$ 

Consider a series $f(x)=\sum_{b\in B} c_b x^b\in\C^{B}$ as a series $f(x)=\sum_{b\in \Z} c_bx^b\in\C^{\Z}$ with $c_b=0$ for $b\notin B.$ 

\begin{defin}\label{def0nondeg}
	In the same notation as above, a pair $(f_1,f_2)$ of series $f_1(x)=\sum_{b\in B_1} c_{1,b} x^b\in\C^{B_1}$ and $f_2(x)=\sum_{b\in B_2} c_{2,b} x^b\in\C^{B_2}$ is {\it $0$-nondegenerate}, if, 
 defining $r_k:=\min (B_1-d_1)\cup(B_2-d_2)\setminus(k\cdot\Z)$, we have $$\dfrac{c_{1,d_1+r_{k}}}{c_{1,d_1}\cdot d_1}\neq\dfrac{c_{2,d_2+r_{k}}}{c_{2,d_2}\cdot d_2}$$ for every common divisor $k>1$ of $d_1$ and $d_2$.
\end{defin}

The proof proceeds by blowing up the origin, given in some affine map as $$\pi:\C^2\to\C^2,\,(t_1,t_2)=\pi(u,v):=(uv,u).$$
The sought intersection number $I$ equals the intersection number of the pull-back divisors $F_i\circ\pi=0$. The latter divisor consists of the exceptional part $u=0$ with multiplicity $d_i-1$ and the germs of smooth reduced curves $C_{i,j},\,j=1,\ldots,d_i$, transversally intersecting $u=0$ at $v=v_{i,j}:=\sqrt[d_i]{1}^j$. Thus the sought intersection number equals $$(d_1-1)(d_2-1)+\sum_{v_{1,j}=v_{2,k}}C_{1,j}\circ C_{2,j}.$$
Thus it remains to find the order of tangency $C_{1,j}\circ C_{2,j}$ of any pair of the branches passing through the same point $v_{1,j}=v_{2,k}$ of the exceptional divisor. This is done in the following lemma.

We denote $g_m(v)=\frac{v^{m}-1}{v-1}$ and for the branch $C_{i,j}$ we set $$m_{i,j}=\min(\{m\in\Z_{\geqslant 0}\mid d_i+m\in B_i,~g(v_{i,j})\neq 0\}).$$
\begin{lemma}
If $(f_1,f_2)$ is 0-nondegenerate, then the order of tangency of any branches $C_{1,j}$ and $C_{2,k}$ at the point $v_0:=v_{1,j}=v_{2,k}$ is equal to $\min(m_{1,j},m_{2,k}).$

Otherwise, the order of tangency of some pairs of branches is strictly greater.
\end{lemma}
\begin{proof}
Note that $F_i\circ\pi(v_0,u)=\alpha_i u^{a_i}+($higher order terms$)$, where $a_1=m_{1,j},~a_2=m_{2,k}$ and $\alpha_i=\frac{-c_{i,d_i+a_i}g_{a_i}(v_0)}{c_{i,d_i}g'_{d_i}(v_0)}.$ Thus the branch $C_{i,j}$ passing through $v_0$ is the graph of a function $v=v_0+\frac{-c_{i,d_i+a_i}g_{a_i}(v_0)}{c_{i,d_i}g'_{d_i}(v_0)}\cdot u^{a_i}+($higher order terms$)$.

Thus the order of tangency equals at least $\min(a_1,a_2)$. Moreover, the equality is attained with no assumptions on the coefficients of $(f_1,f_2)$ in the case $a_1\ne a_2$ and exactly with the 0-nondeneracy assumption in the case $a_1=a_2$. 
\end{proof}

\subsection{Proof of Theorem \ref{thsparse}.3.}

Consider the space $\C^{B_i}$ of all analytic germs $f_i:(\C,0)\to(\C,0)$ whose Taylor series is supported at $B_i\subset\Z_{> 0}$.
Theorem \ref{thsparse}.1\&2 implies the existence of $M>0$ and an algebraic hypersurface $D\subset\C^{B_1\cap[0,M]}\oplus\C^{B_2\cap[0,M]}$, such that the Milnor number is the same for all germs $f\in\C^{B_1}\oplus\C^{B_2}$ outside the preimage of $D$ under the forgetful map $$\C^{B_1}\oplus\C^{B_2}\to\C^{B_1\cap[0,M]}\oplus\C^{B_2\cap[0,M]}.$$
Thus any two such $f$ can be included in a family $f_s$ in which the Milnor number (and hence the topological type by \cite{lr}) remains constant.

\section{Ultratropicalizations, and $\delta$-invariants of algebraic knot diagrams} \label{s2}
A finite set $A\subset\Z^n$ defines the space $\C^A$ of sparse polynomials supported at $A$, i.e. having the form $\sum_{a\in A} c_ax^a$ (where $x^a$ abbreviates the multivariate monomial $x_1^{a_1}\cdot\ldots\cdot x_n^{a_n}$). A sparse polynomial can be regarded as a function on the algebraic torus $\CC^n$.

\subsection{Nested boxes construction and tangency matrices}\label{subs:tangency_matrix}

Consider the lattice $\Z^{n+k}$ with the standard basis $\{e_1,\ldots,e_n,e_{n+1},\ldots, e_{n+k}\}.$ 

\begin{defin} Let $A=(A_1,\ldots,A_q)$ be a collection of finite sets in $\Z^{n+k}$.

1. By $\langle A\rangle$ we denote the minimal sublattice of $\Z^{n+k}$ to which each of $A_i$'s can be parallelly translated.

2. For a surjective linear function $\gamma:\Z^{n+k}\to\Z$ and $d\in\Z_{\geqslant 0}$, the {\it $d$-th crop of $A_m$ in the direction of $\gamma$} is 
$$A_{m,d}^\gamma:=\{a\in A_m,\, \gamma(a)\geqslant\max\gamma(A_m)-d\}.$$
The tuple of the $d$-th crops for all $m=1,\ldots,q$ will be denoted by $A_d^\gamma$.
\end{defin}

\begin{defin}
Let $B$ be a tuple of finite sets in $\Z^{n+k}.$ Denote by $\Lambda_B$ the image of $\langle B\rangle$ under the projection $\rho\colon\Z^{n+K}\to\bigslant{\Z^{n+k}}{\langle e_{n+1},\ldots, e_{n+k}\rangle}.$
The number $\ind_v(B)$ is defined as the index of the sublattice $\Lambda_B$ in $\bigslant{\Z^{n+k}}{\langle e_{n+1},\ldots, e_{n+k}\rangle}.$
\end{defin}

From now on we assume $q=n+k-1.$ Moreover, without loss of generality, we can make the following assumption. 
\begin{predpol}\label{indv1}
The sets $A_m, ~1\leqslant m\leqslant q,$ contain $0\in\Z^{n+k}$ and $\ind_v(A)=1.$ 
\end{predpol}

\begin{defin}
With every surjective linear function $\gamma:\Z^{n+k}\to\Z,$ we associate a sequence $\iota^{\gamma}=(i_1^{\gamma},i_2^{\gamma},\ldots)$ defined as follows: $$i_{d+1}^{\gamma}=\ind_v(A_d^{\gamma}),$$ where $A_d^{\gamma}$ is the tuple of the $d$-th crops of the tuple $A$ in the direction of $\gamma.$
\end{defin}

Assumption \ref{indv1} implies that for every $\gamma\colon\Z^{n+k}\to\Z$ the sequence $\iota^{\gamma}$ stabilizes at 1.

We now define the tangency matrix $(g_{i,j}^\delta(A))$ of the collection $A$ in the direction of a surjective linear function $\delta:\Z^k\to\Z$.

Let $\gamma_j:\Z^n\to\Z,\, j=1,\ldots,N$, be the (finitely many) linear extensions of $\delta$ to $\Z^{n+k}$, such that the Minkowski sum $\sum_m A^{\gamma_j}_{m,0}$ is not contained in a codimension 2 affine plane.
\begin{rem}
Note that each $A^\gamma_{m,0}$ is by definition contained in the support hyperplane of the set $A_m$, parallel to $\ker\gamma$, so the sum $\sum_m A^{\gamma_j}_{m,0}$ is as well contained in an affine hyperplane parallel to $\ker \gamma$. We are looking at the finitely many $\gamma$ for which this sum is not contained in a smaller affine plane.
\end{rem}

We will define the tangency matrix using the so-called {\it nested boxes construction}, which we introduce below. 
In the same notation as above, for $j=1,\ldots, N,$ by $\mathcal R_j$ we denote a finite set of cardinality $\MV(\conv(A_{1,0}^{\gamma_j}),\ldots,\conv(A_{q,0}^{\gamma_j})).$ We will first define the construction for each of the sets $\mathcal R_j$ separately. For every $j,$ it will depend on the sequence $\iota^{\gamma_j}=(i_1,i_2,\ldots)$ introduced above. The nested boxes construction works level by level as follows. To construct level 0, we divide the elements of the set $\mathcal R_j$ into $\frac{|\mathcal R_j|}{i_1}$ boxes containing $i_1$ elements each. Level $d+1$ is obtained by dividing the elements of the level $d$ boxes into $\frac{i_d}{i_{d+1}}$ boxes, each containing ${i_{d+1}}$ elements. The procedure ends when at a certain level each of the boxes contains a single element. 

Fix an arbitrary numbering of boxes at every level of the nested boxes construction. Then, each of the elements $r\in \mathcal R_j$ has its own {\it address,} i.e. a finite sequence of integers, constructed as follows. The $(d+1)$-th element of the address is the number of the $d-th$ level box containing the given element. For any two elements $r_1,r_2$ of $\mathcal R$ with addresses $(a_1,a_2,\ldots,a_M)$ and $(b_1,b_2,\ldots, b_M),$ we define the {\it depth} of their relation as the number $\kappa(r_1,r_2)=\min(\{K\in\N\mid a_K\neq b_K\}).$ Finally, we choose an arbitrary numbering of elements in $\mathcal R_j$ and construct a symmetric $(|\mathcal R_j|\times |\mathcal R_j|)$--matrix $g^j$ by 
\begin{equation*}
g^j_{k,l}=
\begin{cases}
\kappa(r_k,r_l), \text{ if } k\neq l,\\
0,\text{ otherwise}.
\end{cases}
\end{equation*}

\begin{predpol}\label{same_proj}
For any $\delta\colon\Z^k\to\Z$ and for any of its linear extensions $\gamma\colon\Z^{n+k}\to\Z$ such that the sum $\sum_m A^{\gamma_i}_{m,0}$ is not contained in a codimension 2 affine plane, the tuple $A_0^{\gamma}$ satisfies the following property. 
For any $J\subset\{1,\ldots,q\},$ if the sets $A^{\gamma_i}_{m,0}$ in the subtuple $(A^{\gamma_i}_{m,0}\mid m\in J)$ can be shifted to a $|J|$-dimensional subspace $L,$ then $L$ does not contain $\ker\delta.$ 
\end{predpol}
\begin{rem}
This is for instance satisfied if the convex hulls of all $A_i$'s have the same dual fan.
\end{rem}
\begin{defin}\label{defcombtangmatr}
For a tuple $A$ satisfying Assumption \ref{same_proj} the {\it tangency matrix} $g^{\delta}$ is defined to be the block diagonal matrix $\bigoplus\limits_{1\leqslant j\leqslant N}g^{j}.$
\end{defin}
\begin{lemma}
 Under Assumption \ref{same_proj}, the sum of the elements of the tangency matrix $(g_{i,j}^\delta(A))$ equals $$G^\delta(A):=\sum_{l=1}^{N}\sum_{d=1}^{\infty} i^{\gamma_l}_1(i^{\gamma_l}_d-1).$$
\end{lemma}

\subsection{The sum of $\delta$-invariants of a generic sparse curve projection: the answer} 

Consider a tuple $A=(A_1,\ldots, A_{n+1})$ of sets in $\Z^{n+2}$ satisfying Assumptions \ref{indv1} and \ref{same_proj}. In the same notation as in previous subsections, we have $k=2$ and $q=n+2-1=n+1.$

\begin{defin}\label{def:hor_facet}
Let $P$ be a convex polytope of full dimension in $\R^{n+k},$ and $Q\subset\R^n$ be its image under the projection $\rho\colon\R^{n+k}\twoheadrightarrow\bigslant{\R^{n+k}}{\langle e_{n+1},\ldots,e_{n+k}\rangle}.$ A facet $\Gamma\subset P$ is called {\it horizontal}, if $\rho(\Gamma)$ is contained in the boundary of $Q.$
\end{defin}

\begin{defin}\label{def:hor_tuple}
A tuple of the form $A_0^{\gamma}=(A^{\gamma}_{1,0},\ldots,A^{\gamma}_{q,0})$ is called {\it horizontal,} if the Minkowski sum $\conv(A^{\gamma}_{1,0})+\ldots+\conv(A^{\gamma}_{q,0})$ is a horizontal facet of the polytope $\conv(A_1)+\ldots+\conv(A_{q}).$ We denote by $\mathcal H(A)$ the set of all horizontal tuples of $A.$
\end{defin}

\begin{rem}
Equivalently, the set $\mathcal H(A)$ of horizontal tuples in $A$ consists of all tuples $A_0^{\gamma}$ with $\gamma=(\gamma_1,\ldots,\gamma_n,0,\ldots,0).$
\end{rem}

We need the following combinatorial definitions to study the curve $g_1=\ldots=g_{n+1}=0$ given by generic equations $(g_1,\ldots,g_{n+1})\in\C^{A_1}\oplus\cdots\oplus\C^{A_{n+1}}$. Denote the closure of its image under the projection $\CC^{n+2}\to\CC^2$ forgetting $n$ first coordinates by $\mathcal C$. We wish to compute the sum of the $\delta$-invariants of $\mathcal C$.

Let $J_1$ and $J_2$ be the two embeddings $\Z^{n+2}\to\Z^{n+2}\oplus\Z^{n+2}$ mapping to the two direct summands, and $j_1$ and $j_2$ be their compositions with the quotient map $$\Z^{n+2}\oplus\Z^{n+2}\to(\Z^{n+2}\oplus\Z^{n+2})/(J_1+J_2)(\Z^2)\simeq\Z^{2(n+1)}.$$
\begin{theor}\label{thsum}
The sum of the $\delta$-invariants of $C$ equals one half of \begin{multline*}
\MV(j_1A_1,\ldots,j_1A_{n+1},j_2A_1,\ldots,j_2A_{n+1})-\MV(A_1,\ldots,A_{n+1},A_1+\cdots+A_{n+1})+\\+\sum_{A^{\gamma}_0\in\mathcal H(A)}\MV(A^{\gamma}_{1,0},\ldots,A^{\gamma}_{n+1,0})-\sum_{\delta}G^\delta(A).
\end{multline*}
\end{theor}
The latter sum is taken over all surjective linear functions $\delta:\Z^2\to\Z$ and makes sense because $G^\delta(A)$ vanishes for all but finitely many $\delta$. Most of the rest of this section is devoted to the proof of this formula.

\subsection{Puiseux roots and ultratropicalizations} For an algebraic curve $S$ in $ \C^k$ and a sufficiently large $d\in\Z$, the preimage of $S$ under the map $$\C^k\to\C^k,\,(y_1,y_2,\ldots,y_k)\mapsto(y_1^{d!},y_2,\ldots,y_k)\eqno{(*)}$$ splits locally near the hypersurface $y_1=0$ into a union of smooth analytic curves $S_1,\ldots,S_m$ transversal to $y_1=0$.
\begin{defin}\label{defbranch}
The germs of the curves $S_1,\ldots,S_m$ are said to be the {\it Puiseux roots} of $S$ with respect to the divisor $y_1=0$.
\end{defin}
\begin{rem}
1. This notion does not depend on the choice of $d$: the Puiseux roots for different admissible $d$ are in a natural one-to-one correspondence.

2. Puiseux roots are often defined in terms of (multivalued) Puiseaux series expressing the coordinates $(y_2,\ldots,y_k)$ in terms of $y_1$ on $C$, but we prefer to express the same idea in terms of the map $(*)$.

3. The number of Puiseux roots equals the total intersection index of $S$ and $y_1=0$.
\end{rem}
\begin{defin}\label{defdij}
The tangency matrix of $S$ with respect to the divisor $y_1=0$ is the symmetric $m\times m$ matrix with entries $g_{i,i}=0$ and
$$g_{i,j}=(\mbox{the order of tangency of }S_i\mbox{ and }S_j)/d!,\,i\ne j=1,\ldots,m;$$
here the order of tangency is calibrated to equal 1 for transversally intersecting branches and 0 for non-intersecting branches (in some sources it is called the order of contact or the order of coincidence).
\end{defin}
\begin{rem}
1. This notion does not depend on the choice of $d$.

2. If the curve $S$ is locally irreducible at $y_1=0$, then the entries $g_{i,j}$ are known as Puiseux characteristics, and the following equality is classical.
\end{rem} 
\begin{utver}\label{deltamulti}
For $k=2$, the total $\delta$-invariant of the singularities of $S$ at $y_1=0$ equals one-half the sum of the entries of the tangency matrix.
\end{utver}
\begin{proof} By the preceding remark, the statement is invariant under taking the preimage of the curve with respect to the map $(*)$, so, choosing $d$ large enough, we may assume that $S_1,\ldots,S_m$ are the locally irreducible components of the curve at $y_1=0$, and all of them are smooth.
Perturb these branches near $y_1=0$ in such a way that for every pair of branches, their intersection splits into transverse intersections. The total number of those intersections, which is exactly the $\delta$-invariant of $s,$ is equal to the sum of the orders of tangencies of the branches of the singularity. The latter is equal to the sum of entries $g_{i,j}$ with $i<j$ and $S_i, S_j$ being branches of $s.$ Thus the sought total sum of the $\delta$-invariants of the singularities of $S$ at $y_1=0$ is equal to the sum of elements in the upper triangle of the tangency matrix. Finally, this matrix is symmetric, therefore, the sum of elements in its upper triangle is equal to one half of the total sum of the elements of this matrix.  
\end{proof}
\begin{rem}
For $k\geqslant 3$, the tangency matrix does not determine the $\delta$-invariant: e.g. the union of three lines passing through 0 in $\C^3$ has $\delta$-invariant 3 if they are in the same plane and 2 otherwise.
\end{rem}

An algebraic curve $C\subset\CC^k$ admits a tropical compactification, i.e. a smooth toric variety $X_\Sigma\supset\CC^n$ in which the closure $\bar C$ intersects only codimension 1 orbits $O_R$ corresponding to the rays $R$ of the fan $\Sigma$. The tropicalization $\Trop C$ is the union of these rays $R\in\Sigma$ with the multiplicities 
$$m_R:=(\mbox{the intersection number of }\bar C\mbox{ and } O_R).$$
\begin{defin}\label{deftangency}
The ultratropicalization ${\rm UTrop}\, C$ of the curve $C$ is the union of the rays $R$ assigned with the tangency matrices $(g_{i,j}^R)$ of the curve $\bar C$ with respect to the hypersurfaces $\bar O_i$.
\end{defin}
\begin{rem}
Note that every orbit $O_i$ is contained in a suitable toric chart $\C^k\simeq U\subset X_\Sigma$, and the closure $\bar O_i$ is a coordinate hyperplane in this chart. Strictly speaking, we apply the notion of the tangency matrix to the closure of the curve $C\subset\CC^k\subset U$ in this chart (and the result obviously does not depend on the choice of the chart).
\end{rem}

The relation of the ultratropicalization to combinatorial tangency matrices from Definition \ref{defcombtangmatr} is as follows.
\begin{lemma}\label{projtangency}
In the setting of Theorem \ref{thsum}, the tangency matrix $g_{i,j}^R$ of the sparse curve projection $C$ equals the combinatorial tangency matrix $g_{i,j}^\delta(A)$ where $\delta$ is the generator of the ray $R$.
\end{lemma}

\subsection{Proof of Lemma \ref{projtangency}} \label{ssutropproj}
Before we proceed with the proof, let us do some preparatory work that will allow us to simplify the computation. Fix a standard basis $\{e_1,\ldots,e_{n+k}\}$ of $\Z^{n+k}$. For every $N\in\N$ consider the map $F_N\colon\Z^{n+k}\to\Z^{n+k}$ defined as follows: $F_N(e_j)=e_j,~1\leqslant j\leqslant n,$ and $F_N(e_j)=N\cdot e_j,~n+1\leqslant n+k.$ 

The following statement is a straightforward generalization of Lemma 4.4 in \cite{voorh}. 
\begin{lemma}\label{check_A}
In the same notation as above, there exists a number $K\in \N$ such that the images $\check{A_1},\ldots,\check{A_q}$ of the sets $A_1,\ldots, A_q$ under the maps $F_{N}$ with $N>K$ satisfy the following property. If $\gamma=(\gamma_1,\ldots,\gamma_{n+k})\in(\Z^{n+k})^*$ is a primitive covector such that the tuple $\check{A}^{\gamma}_{m,0},~1\leqslant m\leqslant q,$ is not contained in a codimension $2$ affine plane, then the covector $(\gamma_{n+1},\ldots,\gamma_{n+k})\in(\Z^{k})^*$ is also primitive. 
\end{lemma}

\begin{utver}\label{vert_change}
If the sets $A_1,\ldots, A_q$ satisfy the property from Lemma \ref{check_A}, for any tuple of the form $A^{\gamma}_{1,0},\ldots,A^{\gamma}_{q,0}$ that is not contained in a codimension $2$ affine plane, there exists a basis  $\{h_1,\ldots,h_{n+k}\}$ of $\Z^{n+k}$ such that the corresponding change of basis preserves the set $\mathcal H(A)$ of horizontal tuples in $A$ (see Definition \ref{def:hor_tuple}) and in the new basis the tuple $A^{\gamma}_0$ is contained in the hyperplane $\{h_{n+k}=0\}.$
\end{utver}

Let $\Omega$ be the normal fan of the Minkowski sum $\conv(A_1)+\ldots+\conv(A_q).$ Consider the compactification $\widebar{V}\subset X_{\Omega}$ of the complete intersection curve $V.$ 

The main idea of the proof is generalizing a similar computation done in \cite{voorh} to the case when the sets $A_i$ are not necessarily the same.

Using the map {$F_{N!}$} for a sufficiently large {$N,$} we obtain a tuple of supports $\check{A_1},\ldots,\check{A_q}$ satisfying the condition from Lemma {\ref{check_A}}. On one hand, this step simplifies the computation, since Proposition \ref{vert_change} becomes applicable. On the other hand, the tangency matrices for the tuple $\check{A_1},\ldots,\check{A_q}$ (see Definition \ref{defdij}) are closely related to those for the initial tuple. Namely, the blocks of the tangency matrices for the tuple $\check{A_1},\ldots,\check{A_q}$ consist of blocks of the corresponding tangency matrices for the initial tuple multiplied by $N!$ and repeated $N!$ times. Thus the total sums of the entries of the tangency matrices for the tuple $\check{A_1},\ldots,\check{A_q}$ differ from those for the initial tuple by the factor $(N!)^2.$

First, we observe that Assumption \ref{same_proj} guarantees that for generic $(f_1,\ldots,f_{n+1})\in\C^{A_1}\times\ldots\times\C^{A_{n+1}},$ for any ray $R$ of the fan $X_{\Sigma}$ and any pair $\gamma_1,\gamma_2\in(\R^{n+2})^*$ of distinct linear extensions of the primitive generator $\delta$ of $R$ satisfying $\MV(A^{\gamma_i}_{1,0},\ldots A^{\gamma_i}_{n+1,0}) \neq 0, ~i=1,2,$ the following equality holds: 
\begin{equation}\label{eq:dif_roots}
\pi(\{f_1^{\gamma_1}=\ldots=f_{n+1}^{\gamma_1}=0\})\cap\pi(\{f_1^{\gamma_1}=\ldots=f_{n+1}^{\gamma_1}=0\})=\emptyset.
\end{equation}

Indeed, by Assumption \ref{same_proj}, for any $j\in\{1,\ldots,n+1\},$ the projection $\pi$ maps the curve $Y_j=\{f_1^{\gamma_1}=\ldots=\widehat{f_{j}^{\gamma_1}}=\ldots=f_{n+1}^{\gamma_1}=0\}$ (the $j$-th polynomial in the system defining $Y_j$ is omitted) surjectively onto $O_R.$ Therefore, for any $j\in\{1,\ldots,n+1\},$ the set $\pi^{-1}(\mathcal P)\cap Y_j,$ where $\mathcal P=\pi(\{f_1^{\gamma_2}=\ldots=\ldots=f_{n+1}^{\gamma_2}=0\}\subset O_R,$ is finite. 

Since $\gamma_1\neq\gamma_2,$ for some index $j_0\in\{1,\ldots,n+1\},$ the difference $A^{\gamma_1}_{j_0,0}\setminus A^{\gamma_2}_{j_0,0}$ is non-empty. Take any point $a\in A^{\gamma_1}_{j_0,0}\setminus A^{\gamma_2}_{j_0,0}.$ 

Finally, note that since $\MV(A^{\gamma_1}_{1,0},\ldots A^{\gamma_1}_{n+1,0}) \neq 0,$ by Bernstein's criterion for vanishing of the mixed volume, we have $f_{j_0}/x^a\mid_{Y_{j_0}}\neq\mathrm{const}.$ 
The latter implies that for generic $\alpha\in\C,$ we have $f_{j_0}/x^a\mid_{\pi^{-1}(\mathcal P)\cap Y_{j_0}}\neq\alpha.$ Thus, perturbing the coefficient at the monomial $x^a$ in the polynomial $f_{j_0},$ we can always enforce \ref{eq:dif_roots} to hold. 

Let $p$ be a multiple point of intersection of the closure of the curve $\mathcal C$ with the $1$-dimensional orbit $O_R$ of the toric variety $X_{\Sigma}$ corresponding to the ray $R.$ In a small neighborhood of the point $p,$ the curve $C$ is a union of smooth branches intersecting $O_R$ transversally (upon a change of coordinates from the preceding lemma). We will now compute the order of contact between any two of those branches in terms of the support sets $A_1,\ldots, A_{n+1}.$ 

After performing a monomial change of variables from Proposition \ref{vert_change}, one can write the polynomials $f_i,~1\leqslant\ldots i\leqslant n+1,$ in the following form: $$f_i=g_i(x_1\ldots,x_n,y_1)+\sum_{m=1}^{\infty} y_2^m\widetilde{g}_{i,m}(x_1\ldots,x_n,y_1),$$ where $g_i=f_i^{\gamma}$ and $\widetilde{g}_{i,m}$ are polynomials in the variables $x_1\ldots,x_n,y_1.$ The orbit of the toric variety $X_{\Omega}$ corresponding to the tuple $A^{\gamma}_{0}$ is the coordinate hyperplane $\{y_2=0\}.$

Let $p=(0,\ldots,0,u_1,0)$ be a multiple intersection point of the curve $C$ with the $1$-dimensional orbit of the toric variety $X_{\Sigma}$ corresponding to the ray $R.$ Consider a pair of branches of the curve $C$ passing through the point $p.$ The preimages of these two branches intersect the hyperplane $\{y_2=0\}$ at the points $p_1=(v_1,\ldots,v_n,u_1,0), p_2=(v'_1,\ldots,v'_n,u_1,0)$ respectively. Moreover, for every $1\leqslant i\leqslant q,$ we have $g_i(p_1)=g_i(p_2).$ Due to Assumption \ref{indv1} the tuple $(A_1,\ldots,A_{n+1})$ satisfies $\ind_v(A_1,\ldots,A_{n+1})=1$. Therefore there exists a number $\kappa\in\N,$ such that for $1\leqslant i\leqslant n+1,$ we have $\widetilde{g}_{i,\kappa-1}(p_1)=\widetilde{g}_{i,\kappa-1}(p_2)$ and $\widetilde{g}_{i,\kappa}(p_1)\neq\widetilde{g}_{i,\kappa}(p_2).$ We denote this number by $\kappa(p_1,p_2)$ to emphasize its dependence on the points $p_1,p_2.$ 

\begin{proof}[Proof of Lemma \ref{projtangency}]
	Let $\kappa=\kappa(p_1,p_2),$ where $p_1,p_2$ are as described above. Let us first compute the lower bound for the sought order of contact using the following proposition.
	\begin{utver}\label{lowerbound}
		The order of contact between the projections of the branches of $C$ passing through the points $p_1,p_2$ is greater or equal to $\kappa.$ 
	\end{utver} 
	\begin{proof}[Proof of Proposition \ref{lowerbound}]
The points $p_1=(v_1,\ldots,v_n,u_1,0)$ and $p_2=(v'_1,\ldots,v'_n,u_1,0)$ under consideration are related as follows. The lattice $\Z^n\simeq\bigslant{\Z^{n+2}}{\langle h_{n+1},h_{n+2}\rangle}$ and the lattice $\Lambda\subset\Z^n$ affinely generated by the tuple $A^{\gamma}_0$ admit a pair of aligned bases $\Z^n=\bigoplus\Z w_i$ and $\Lambda=\bigoplus\Z a_iw_i,$ for some $a_1,\ldots, a_n\in\Z.$ Therefore, for some integers $a_1,\ldots,a_n,$ we have $v'_j=r_jv_j,$ where $r_j$ is some $(a_j)$-th root of unity. Also, note that rescaling the coordinate system in the same manner, i.e. performing the change of variables $\check{x}_i=r_jx_j,$ does not affect the coefficients of monomials of $f_1,\ldots,f_{n+1}$ that do not distinguish the points $p_1$ and $p_2.$ For $1\leqslant j\leqslant q,$ denote by $F_j$ the part of $f_j$ which is invariant under this rescaling, and by $G_j$ the part that is not. 
		
		Now, let us compute the order of contact at the point $p_1$ between the complete intersection curves $\mathcal C=\{f_1=\ldots=f_{n+1}=0\}$ and $X=\{F_1=\ldots=F_{n+1}=0\}.$ Let the curve $X$ be locally parametrized, i.e., in a neighborhood of the point $p_1,$ the curve $X$ is the image of a parametrization map $s\mapsto\varphi(s)=(\varphi_1(s),\ldots,\varphi_n(s),\varphi_{n+1}(s),\varphi_{n+2}(s)).$ Note also that since the $(n+2)$-th coordinate of $p_1$ is $0,$ the Taylor series of $\varphi_{n+2}$ at $p_1$ has no constant term, therefore it starts with the term of degree at least $1$ in $s.$ 
		
		Substituting $\varphi(s)$ into the system $\{f_1=\ldots f_{n+1}=0\},$ we obtain the following system: $$\{G_1(\varphi(s))=\ldots=G_{n+1}(\varphi(s))=0\}.$$ The way the number $K=K(p_1,p_2)$ was defined implies that the polynomials $G_j$ are of the form $G_j=\sum_{m=K}^{\infty} t^m\widetilde{g}_{i,m}(x_1\ldots,x_n,y_1).$ Thus, for $1\leqslant j\leqslant n+1,$ the leading term in $G_j(\varphi(s))$ is of degree $\geqslant K$ in $s.$ Therefore, the order of contact between $X$ and the curve $V$ at $p_1$ is at least $K.$ The same is true for the order of contact between $X$ and $V$ at $p_2.$ 
		
		Now we note that the change of variables $\check{x}_i=r_jx_j$ maps the branch of the curve $X$ passing through $p_1$ to the one passing through $p_2$ and at the same time does not change the defining polynomials of $X.$ Moreover, since this change of variables does not affect the last two of the $n+2$ coordinates, the projections of the two above-mentioned branches of $X$ coincide in some neighborhood of $p=\pi(p_1)=\pi(p_2).$ Finally, the projection $\pi$ does not decrease the order of contact between curves, which yields the desired inequality. 
	\end{proof}
	Now let us state and prove the following special case of Proposition \ref{projtangency}.
	\begin{utver}\label{transverse}
		In the same notation as above, suppose that $\kappa(p_1,p_2)=1.$ Then for generic $f_1,\ldots,f_{n+1}\in\C^{A_1}\oplus\ldots\oplus \C^{A_{n+1}},$ the projections of the branches of the curve $C$ that pass through $p_1$ and $p_2$ intersect transversally. 
	\end{utver}

	\begin{proof}[Proof of Proposition \ref{transverse}]
	To prove the statement we need to compare the projections of the tangent lines to the curve $V$ at the points $p_1$ and $p_2$ and make sure that for almost all $f_1,\ldots,f_{n+1}\in\C^{A_1}\oplus\cdot\oplus\C^{A_{n+1}},$ they do not coincide. 
	
	To find the tangent lines at $p_1$ and $p_2$ we need to compute the kernel of the differential $n+1$-form $\bigwedge_{i=1}^{q} df_i$ at those points. Moreover, since we need to compare the projections of the tangent lines, the only components of $\bigwedge_{i=1}^{n+1} df_i$ that we have to look at are $dx_1\wedge\ldots\wedge dx_n\wedge dy_1$ and $dx_1\wedge\ldots\wedge dx_n\wedge dy_2.$ In the setting that is being considered, the corresponding coefficients will be the  maximal minors $M_{1},M_{2}$ obtained by keeping the first $n$ columns of the $(n+1)\times(n+2)$-matrix $\mathcal J$ evaluated at the points $p_1$ and $p_2,$ where 
	
	$$\mathcal J=\begin{pmatrix} 
		\frac{\partial g_1}{\partial x_1}&\dots&\frac{\partial g_1}{\partial x_n}&\frac{\partial g_1}{\partial y_1}&\widetilde{g}_{1,1}\\
		\frac{\partial g_2}{\partial x_1}&\dots&\frac{\partial g_2}{\partial x_n}&\frac{\partial g_2}{\partial y_1}&\widetilde{g}_{2,1}\\
		\vdots & \ddots & \vdots & \vdots & \vdots \\
		\frac{\partial g_{q}}{\partial x_1}&\dots&\frac{\partial g_{n+1}}{\partial x_n}&\frac{\partial g_{n+1}}{\partial y_1}&\widetilde{g}_{n+1,1}\\
	\end{pmatrix}.$$
	
Since $K(p_1,p_2)=1,$ at least one of the polynomials $\tilde{g}_{j,1}$ has a monomial which distinguishes the points $p_1,p_2.$ Without loss of generality, suppose that such a monomial occurs in the polynomial $\tilde{g}_{1,1}.$ Moreover, since the curve $V$ intersects the orbit of $X_{\Omega}$ corresponding to the tuple $(A^{\gamma}_{1,0},\ldots,A^{\gamma}_{n+1,0})$ transversally, the minor $M_{1}$ does not vanish at $p_1$ and $p_2.$ 

Therefore, we need to show that for generic $f_1,\ldots,f_q,$ we have 
\begin{equation}\label{jac}
\begin{vmatrix}
			M_{1}(p_1)&M_{2}(p_1)\\
			M_{1}(p_2)&M_{2}(p_2)\\
\end{vmatrix}\neq 0.
	\end{equation}

The condition (\ref{jac}) is clearly algebraic, so, the set of tuples $f_1,\ldots,f_{n+1}\in\C^A$ satisfying it is Zariski open. To make sure it is everywhere dense, we need to show that it is non-empty. Indeed, suppose that for some $f_1,\ldots,f_{n+1}\in\C^A,$ we have  
\begin{equation}\label{tangentcond}
		\begin{vmatrix}
			M_{1}(p_1)&M_{2}(p_1)\\
			M_{1}(p_2)&M_{2}(p_2)\\
		\end{vmatrix}= 0.
\end{equation}
Since the minor $M_{1}$ does not vanish at the points $p_1$ and $p_2,$ the second column of the matrix in (\ref{tangentcond}) is non-zero and equality (\ref{tangentcond}) can be rewritten as follows: 
	\begin{equation}\label{syst}
		\begin{cases}
			M_{1}(p_1)=\lambda M_{2}(p_1)\\
			M_{1}(p_2)=\lambda M_{2}(p_2),
		\end{cases}
	\end{equation}
	for some $\lambda\in\C.$
	The minors $M_{1}$ and $M_{2}$ are given by the following formulas:
	\begin{equation}\label{Minors}
		M_{2}=\begin{vmatrix} 
			\frac{\partial g_1}{\partial x_1}&\dots&\frac{\partial g_1}{\partial x_n}&\widetilde{g}_{1,1}\\
			\frac{\partial g_2}{\partial x_1}&\dots&\frac{\partial g_2}{\partial x_n}&\widetilde{g}_{2,1}\\
			\vdots & \ddots & \vdots & \vdots \\
			\frac{\partial g_{n+1}}{\partial x_1}&\dots&\frac{\partial g_{n+1}}{\partial x_n}&\widetilde{g}_{n+1,1}\\
		\end{vmatrix},
		~M_{1}=\begin{vmatrix} 
			\frac{\partial g_1}{\partial x_1}&\dots&\frac{\partial g_1}{\partial x_n}&\frac{\partial g_1}{\partial y_{k-1}}\\
			\frac{\partial g_2}{\partial x_1}&\dots&\frac{\partial g_2}{\partial x_n}&\frac{\partial g_2}{\partial y_{k-1}}\\
			\vdots & \ddots & \vdots & \vdots \\
			\frac{\partial g_{n+1}}{\partial x_1}&\dots&\frac{\partial g_{n+1}}{\partial x_n}&\frac{\partial g_{n+1}}{\partial y_{1}}\\
		\end{vmatrix}.
	\end{equation}
	The matrices in (\ref{Minors}) are almost identical except for the last column, let us expand their determinants along this column. Then, we obtain:
	\begin{align}\label{determexpand}
		M_{k}=\sum_{j=1}^{n+1}(-1)^{n+1+j}\widetilde{g}_{j,1}D_j=(-1)^{n+2}\widetilde{g}_{1,1}D_1+\sum_{j=2}^{n+1}(-1)^{n+1+j}\widetilde{g}_{j,1}D_j;\\
		M_{k-1}=\sum_{j=1}^{n+1}(-1)^{n+1+j}\frac{\partial g_j}{\partial y_{k-1}}D_j=(-1)^{n+2}\frac{\partial g_1}{\partial y_{k-1}}D_1+\sum_{j=2}^{n+1}(-1)^{n+1+j}\frac{\partial g_j}{\partial y_{k-1}}D_j, 
	\end{align}
	where $D_j$ are the cofactors of the entries in the last columns.

Assumption \ref{same_proj} guarantees that for generic $f_1,\ldots,f_{n+1},$ the cofactor $D_1$ does not vanish at the roots of the truncated system $\{g_1(x_1,\ldots,x_n,y_1)=\ldots=g_{n+1}(x_1,\ldots,x_n,y_1)=0\}.$ Indeed, consider the curve $Y$ given by the polynomials $g_2,\ldots,g_{n+1}$ in the coordinate hyperplane $\{y_2=0\}.$ Assumption \ref{same_proj} together with Theorem 17 in \cite{khov16} implies that $Y$ is irreducible and is not contained in a shifted copy of a subtorus in $\CC^{n+1}.$ Therefore, at a generic point of $Y$ the $\partial_{y_1}$-component of the vector tangent to $Y$ (which equals the cofactor $D_1$) is non-zero. 

Finally, for generic $f_1,\ldots,f_{n+1}$ the intersection points of the hypersurface $\{g_1=0\}$ and the curve $Y$ are generic, which implies that the cofactor $D_1$ does not vanish at the roots of the 
system $\{g_1(x_1,\ldots,x_n,y_1)=\ldots=g_{n+1}(x_1,\ldots,x_n,y_1)=0\}.$ 

 Substituting equations in (\ref{determexpand}) into the system (\ref{syst}), we obtain:
	\begin{equation}\label{syst1}
		\begin{cases}
			(-1)^{n+2}(\widetilde{g}_{1,1}-\lambda\frac{\partial g_1}{\partial y_{1}})D_1\mid_{p_1}=\sum_{j=2}^{n+1}(-1)^{n+j}(\widetilde{g}_{j,1}-\lambda\frac{\partial g_j}{\partial y_{1}})D_j\mid_{p_1}\\
			(-1)^{n+2}(\widetilde{g}_{1,1}-\lambda\frac{\partial g_1}{\partial y_{1}})D_1\mid_{p_2}=\sum_{j=2}^{n+1}(-1)^{n+j}(\widetilde{g}_{j,1}-\lambda\frac{\partial g_j}{\partial y_{1}})D_j\mid_{p_2}.
		\end{cases}
	\end{equation}

The equality $K=K(p_1,p_2)=1$ implies that the polynomial $\widetilde{g}_{1,1}$ has at least one monomial $\alpha\in\supp\widetilde{g}_{1,1}$ such that $\alpha(p_1)\neq\alpha(p_2).$ Changing the coefficient of this monomial in $\widetilde{g}_{1,1}$ (i.e. adding $\varepsilon\cdot\alpha,~\varepsilon\neq 0$ to $\widetilde{g}_{1,1}$) affects the left-hand sides of each of the equalities in (\ref{syst1}). Moreover, the choice of the monomial $\alpha$ guarantees that  $(-1)^{n+2}(\widetilde{g}_{1,1}+\varepsilon\cdot\alpha-\lambda\frac{\partial g_1}{\partial y_{1}})D_1\mid_{p_1}\neq (-1)^{n+2}(\widetilde{g}_{1,1}+\varepsilon\cdot\alpha-\lambda\frac{\partial g_1}{\partial y_{1}})D_1\mid_{p_2}.$ 
The right-hand sides of the equalities in (\ref{syst1}) do not depend on $\widetilde{g}_{1,1}$ and thus remain unchanged. Therefore, the equalities (\ref{syst1}) become no longer true, so the determinant in (\ref{tangentcond}) does not vanish anymore. 
\end{proof}

To deduce the statement of Lemma \ref{projtangency} for $\kappa=\kappa(p_1,p_2)>1$ from Proposition \ref{transverse} together with Proposition \ref{lowerbound}, we can use the following genericity argument, which, however, does not produce an explicit genericity condition on the coefficients of the polynomials $f_1,\ldots,f_{n+1}$ for $\kappa>1$ (in contrast to the explicit notion of 0-nondegeneracy for parameterized sparse curve singularities). 

If there are no other monomials $\alpha\in \supp(\widetilde{g}_{i,m}), i\neq \kappa,$ then the change of variables $\check{t}=t^{\kappa}$ reduces this case to the one considered above. Moreover, the coefficients of the monomials $\alpha\in \supp(\widetilde{g}_{i,m}),~i> \kappa,$ do not change the sought order of contact. 

We still need to deal with the monomials $\alpha\in \supp(\widetilde{g}_{i,m}),~i<\kappa.$ Let us consider the space parametrized by all possible choices of coefficients of the monomials $\alpha\in \supp(\widetilde{g}_{i,m}),~i<\kappa.$ The sought order of contact is  an integer-valued function on this space, which constructively depends on the choice of these coefficients. Its lower bound is equal to $\kappa,$ which follows from Proposition \ref{lowerbound}. Moreover, if all of these coefficients are equal to $0,$ then the desired order of contact is exactly $\kappa.$ By upper semi-continuity of the intersection index the latter implies that for a generic choice of the coefficients of the monomials $\alpha\in \supp(\widetilde{g}_{i,m}),~i<\kappa,$ the order of contact does not exceed $\kappa,$ and thus, is equal to $\kappa.$

Finally, in the notation of subsection \ref{subs:tangency_matrix}, take $\mathcal R$ to be the set of branches of the curve $\widebar V$ intersecting the orbit of the toric variety $X_{\Omega}$ corresponding to the tuple $A^{\gamma}_0.$ One can identify each element $r_i\in \mathcal R$ with the intersection point $p_i$ of the corresponding branch and $X_{\Omega}.$ Observing that the depth of the relation between a pair of elements $r_i,r_j\in \mathcal R$ in the nested boxes construction for the tuple $A^{\gamma}_0$ is exactly the number $\kappa(p_i,p_j)$ concludes the proof of Lemma \ref{projtangency}. 
\end{proof}

\subsection{Mixed fiber polytopes and the Newton polygon of a sparse curve projection}

Let $L\subset\R^{N}$ be a rational codimension $d$ subspace. On each of the two spaces, we introduce the lattice volume form (i.e. the minimal volume form that takes integer values on every lattice polytope in the respective space) and denote the corresponding mixed volume by $\MV$. 

\begin{defin}\label{defmp}
The mixed fiber body $\MP_L{\Delta_0,\ldots,\Delta_{d}}$ of convex bodies $\Delta_i\subset \R^N$ is the unique convex body $\Delta\subset L$ such that for any tuple $B_1,\ldots, B_{N-d-1}$ of convex bodies in $L$ the following equality holds: 
\begin{equation}
\MV(\Delta_0,\ldots,\Delta_{d},B_1,\ldots, B_{N-d-1})=\MV(\Delta, B_1,\ldots, B_{N-d-1}).
\end{equation}
\end{defin}
\begin{exa} If $L=\{0\}\times\R^2\subset\R^3$, then
$$\MP_L[0,a]\times[0,b]\times\{0\},[0,a]\times\{0\}\times[0,c]=[0,ab]\times[0,ac]\subset L.$$
\end{exa}
The following is shown in \cite{EKh} for polytopes and in \cite{mmj} for convex bodies:

\vspace{1ex}

(1) the mixed fiber body exists, is unique up to a shift, and is a lattice polytope whenever $\Delta_i$'s are lattice polytopes;

(2) as a function of $\Delta_i$'s, it is symmetric, multilinear with respect to the Minkowski summation of arguments, and equals $(d+1)\cdot($fiber body of $\Delta)$ whenever $\Delta_0=\cdots=\Delta_{d}=\Delta$;

(3) the aforementioned three properties uniquely define $\MP_L{\Delta_0,\ldots,\Delta_{d}}$ as a function of $\Delta_i$'s (i.e. it may be regarded as a properly scaled polarization of the fiber body functional);

(4) For generic Laurent polynomials $f_0,\ldots,f_d$ of variables $x_1,\ldots,x_N$ with the Newton polytopes $\Delta_0,\ldots,\Delta_d$, consider the image of the complete intersection $$V:=\{f_0=\cdots=f_d=0\}$$ under the coordinate projection to $$L:=\{x_1=\cdots=x_d=0\}$$ as a Cartier divisor $D$ (with the multiplicities reflecting the degree of the ramified covering $V\to D$). Then the equation of $D$ has the Newton polytope $\MP_L{\Delta_0,\ldots,\Delta_{d}}$.

\vspace{1ex}

In our setting, $N=n+k, d=q-1=n+k-2$ and $L$ is a coordinate plane in $\R^{n+k}.$ The Newton polygon of the curve $C,$ which is the image of a complete intersection curve $V$ given by a generic tuple of polynomials $f_1,\ldots, f_q$ with $\supp(f_i)=A_i$ is equal to the mixed fiber polygon $\MP_L(\conv(A_1),\ldots,\conv(A_q)).$

Proposition \ref{fib_area} below, which is a special case of Proposition 3.40 in \cite{E18}, gives an explicit formula for the area of the mixed fiber polygon $\MP_L(\conv(A_1),\ldots,\conv(A_q)).$ To formulate it, we need to introduce a bit of notation. 

Let $J_1$ and $J_2$ be the two embeddings $\Z^{n+2}\to\Z^{n+2}\oplus\Z^{n+2}$ mapping to the two direct summands, and $j_1$ and $j_2$ be their compositions with the quotient map $$\Z^{n+2}\oplus\Z^{n+2}\to(\Z^{n+2}\oplus\Z^{n+2})/(J_1+J_2)(\Z^2)\simeq\Z^{2(n+1)}.$$

\begin{utver}\label{fib_area}
In the same notation as above, the area of the Newton polygon of the curve $C$ can be computed via the following formula: 
\begin{equation}
\vol(\MP_L(A_1,\ldots,A_q))=\MV(j_1A_1,\ldots,j_1A_{q},j_2A_1,\ldots,j_2A_{q}). 
\end{equation}
\end{utver}

\begin{exa}\label{exa:MFP}
Let us compute the Newton polygon of the sparse curve $C$ supported at the sets $(B_1,B_2)$ from Example \ref{exa0} (see Figure \ref{fig:MFP1}). We denote $h_i=\max(B_i).$ 
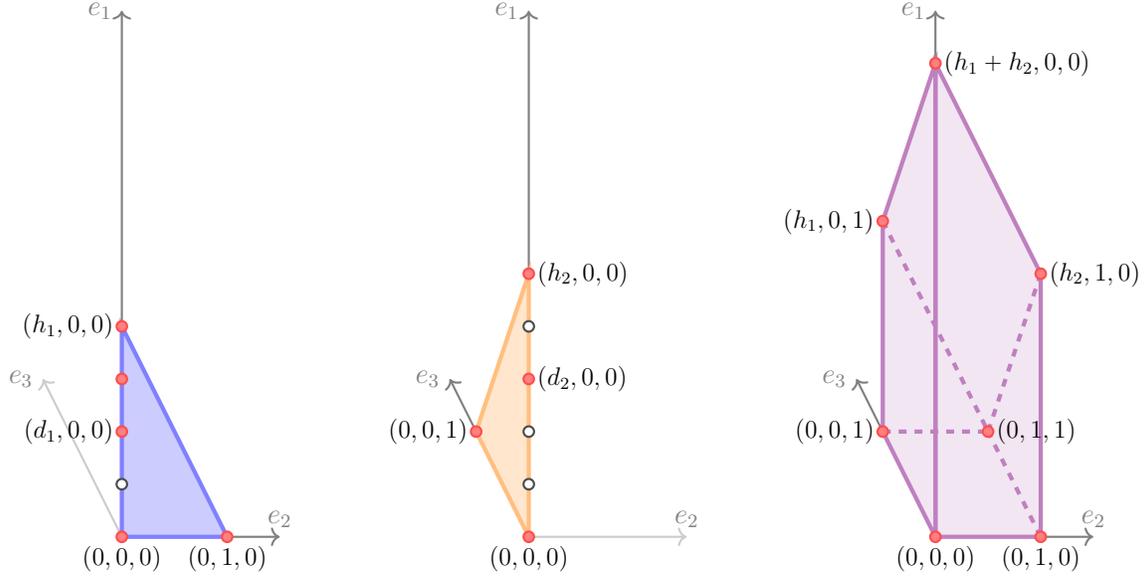
\begin{figure}[ht]
\begin{center}	
		\begin{tikzpicture}[scale=0.7]
			\node[left,gray, scale=0.9] at (0,10) {$e_1$};
			\node[above,gray, scale=0.9] at (3,0) {$e_2$};
			\node[left,gray, scale=0.9] at (-1.5,3) {$e_3$};
			\node[below, scale=0.8] at (0,0) {$(0,0,0)$};
			\node[below, scale=0.8] at (2,0) {$(0,1,0)$};
			\node[left, scale=0.8] at (0,4) {$(h_1,0,0)$};
			\draw[->, thick, gray] (0,0)--(3,0);
			\draw[->, thick, gray] (0,0)--(0,10);
			\draw[->, thick, gray!40] (0,0)--(-1.5,3);
			\node[left, scale=0.8] at (0,2) {$(d_1,0,0)$};
			\filldraw[color=blue!50, fill=blue!20, ultra thick] (0,0)--(2,0)--(0,4)--(0,0);
			\filldraw[color=red!70, fill=red!50, thick](0,0) circle (0.1);
			\filldraw[color=red!70, fill=red!50, thick](2,0) circle (0.1);
			\filldraw[color=red!70, fill=red!50, thick](0,2) circle (0.1);
			\filldraw[color=red!70, fill=red!50, thick](0,3) circle (0.1);
			\filldraw[color=red!70, fill=red!50, thick](0,4) circle (0.1);
			\filldraw[color=black!70, fill=white, thick](0,1) circle (0.1);
		\end{tikzpicture}
		\begin{tikzpicture}[scale=0.7]
			\node[left,gray, scale=0.9] at (0,10) {$e_1$};
			\node[above,gray, scale=0.9] at (3,0) {$e_2$};
			\node[left,gray, scale=0.9] at (-1.5,3) {$e_3$};
			\node[below, scale=0.8] at (0,0) {$(0,0,0)$};
			\node[left, scale=0.8] at (-1,2) {$(0,0,1)$};
			\node[right, scale=0.8] at (0,5) {$(h_2,0,0)$};
			\node[right, scale=0.8] at (0,3) {$(d_2,0,0)$};
			\filldraw[color=white, fill=white, thick](-4,0) circle (0.1);
			\draw[->, thick, gray!40] (0,0)--(3,0);
			\draw[->, thick, gray] (0,0)--(0,10);
			\draw[->, thick, gray] (0,0)--(-1.5,3);
			\filldraw[color=orange!50, fill=orange!20, ultra thick] (0,0)--(-1,2)--(0,5)--(0,0);
			\filldraw[color=red!70, fill=red!50, thick](0,0) circle (0.1);
			\filldraw[color=red!70, fill=red!50, thick](0,5) circle (0.1);
			\filldraw[color=red!70, fill=red!50, thick](-1,2) circle (0.1);
			\filldraw[color=black!70, fill=white, thick](0,2) circle (0.1);
			\filldraw[color=red!70, fill=red!50, thick](0,3) circle (0.1);
			\filldraw[color=black!70, fill=white, thick](0,4) circle (0.1);
			\filldraw[color=black!70, fill=white, thick](0,1) circle (0.1);
		\end{tikzpicture}
		\begin{tikzpicture}[scale=0.7]
			\node[left,gray, scale=0.9] at (0,10) {$e_1$};
			\node[above,gray, scale=0.9] at (3,0) {$e_2$};
			\node[left,gray, scale=0.9] at (-1.5,3) {$e_3$};
			\node[below, scale=0.8] at (0,0) {$(0,0,0)$};
			\node[left, scale=0.8] at (-1,2) {$(0,0,1)$};
			\node[below, scale=0.8] at (2,0) {$(0,1,0)$};
			\node[right, scale=0.8] at (0,9) {$(h_1+h_2,0,0)$};
			\node[left, scale=0.8] at (-1,6) {$(h_1,0,1)$};
			\node[right, scale=0.8] at (2,5) {$(h_2,1,0)$};
			\filldraw[color=white, fill=white, thick](-4,0) circle (0.1);
			\draw[->, thick, gray] (0,0)--(3,0);
			\draw[->, thick, gray] (0,0)--(0,10);
			\draw[->, thick, gray] (0,0)--(-1.5,3);
			\filldraw[color=violet!50, fill=violet!10, ultra thick] (0,0)--(2,0)--(2,5)--(0,9)--(0,0);
			\filldraw[color=violet!50, fill=violet!10, ultra thick] (0,0)--(-1,2)--(-1,6)--(0,9)--(0,0);
			\draw[color=violet!50,ultra thick, dashed] (-1,2)--(1,2);
			\draw[color=violet!50,ultra thick, dashed] (-1,6)--(1,2);
			\draw[color=violet!50,ultra thick, dashed] (2,0)--(1,2);
			\draw[color=violet!50,ultra thick, dashed] (2,5)--(1,2);
			\filldraw[color=red!70, fill=red!50, thick](0,0) circle (0.1);
			\filldraw[color=red!70, fill=red!50, thick](0,9) circle (0.1);
			\filldraw[color=red!70, fill=red!50, thick](-1,2) circle (0.1);
			\filldraw[color=red!70, fill=red!50, thick](-1,6) circle (0.1);
			\filldraw[color=red!70, fill=red!50, thick](2,5) circle (0.1);
			\filldraw[color=red!70, fill=red!50, thick](1,2) circle (0.1);
			\filldraw[color=red!70, fill=red!50, thick](2,0) circle (0.1);
			\node[right, scale=0.8] at (1,2) {$(0,1,1)$};
		\end{tikzpicture}
	\end{center}
	\caption{The polygons $\conv(A_1)$, $\conv(A_2)$ and their Minkowski sum.}\label{fig:MFP1}
\end{figure}
The fiber polygons $I_1, I_2, P$ of $\conv(A_1),\conv(A_2)$ and $\conv(A_1)+\conv(A_2)$ respectively are shown in Figure \ref{fig:MFP2} below. The sought Newton polygon is the mixed fiber polygon of $\conv(A_1)$ and $\conv(A_2),$ and the latter (by its definition) equals $\frac{P-I_1-I_2}{2}.$ The result (as shown in Figure \ref{fig:MFP2}) is the triangle with the vertices $(0,0),(0,h_1)$ and $(h_2,0).$
\begin{figure}[ht]
\begin{center}	
\begin{tikzpicture}[scale=0.4]
\draw[->, thick, gray] (0,0)--(6,0);
\draw[->, thick, gray] (0,0)--(0,6);
\node[above,gray, scale=0.9] at (0,6) {$e_3$};
\node[above,gray, scale=0.9] at (6,0) {$e_2$};
\draw[ultra thick, blue] (0,0)--(4,0);
\node[below, scale=0.65] at (4,0) {$(h_1,0)$};
\node[below, scale=0.65] at (0,0) {$(0,0)$};
\filldraw[color=red!70, fill=red!50, thick](0,0) circle (0.15);
\filldraw[color=red!70, fill=red!50, thick](4,0) circle (0.15);
\node[above, blue,scale=0.9] at (2,0) {$I_1$};
\end{tikzpicture}
\begin{tikzpicture}[scale=0.4]
\draw[->, thick, gray] (0,0)--(6,0);
\draw[->, thick, gray] (0,0)--(0,6);
\node[above, gray, scale=0.9] at (0,6) {$e_3$};
\node[above, gray, scale=0.9] at (6,0) {$e_2$};
\draw[ultra thick, orange] (0,0)--(0,5);
\node[left, scale=0.65] at (0,5) {$(0,h_2)$};
\node[below, scale=0.65] at (0,0) {$(0,0)$};
\filldraw[color=red!70, fill=red!50, thick](0,0) circle (0.15);
\filldraw[color=red!70, fill=red!50, thick](0,5) circle (0.15);
\node[right, orange ,scale=0.9] at (0,2.5) {$I_2$};
\end{tikzpicture}
\begin{tikzpicture}[scale=0.3]
\draw[->, thick, gray] (0,0)--(18.5,0);
\draw[->, thick, gray] (0,0)--(0,15);
\node[above,gray, scale=0.9] at (0,15) {$e_3$};
\node[above,gray, scale=0.9] at (18,0) {$e_2$};
\draw[ultra thick,violet] (0,0)--(0,13)--(4,13)--(14,5)--(14,0)--(0,0);
\pattern[pattern=vertical lines, pattern color=orange] (0,8)--(0,13)--(4,13)--(14,5)--(14,0)--(4,8)--(0,8);
\pattern[pattern=horizontal lines, pattern color=blue] (0,8)--(4,8)--(14,0)--(10,0)--(0,8);
\draw[ultra thick,blue!80] (0,8)--(10,0);
\draw[ultra thick,orange!80] (0,8)--(4,8)--(14,0);
\draw[ultra thick,violet!80] (0,0)--(0,13)--(4,13)--(14,5)--(14,0)--(0,0);
\node[left, scale=0.65] at (0,8) {$(0,2h_1)$};
\node[below, scale=0.65] at (10,0) {$(2h_2,0)$};
\node[below right, scale=0.65] at (13,0) {$(2h_2+h_1,0)$};
\node[right, scale=0.65] at (14,5) {$(2h_2+h_1,h_2)$};
\node[below, scale=0.65] at (0,0) {$(0,0)$};
\node[above, scale=0.65] at (4,13) {$(h_1,2h_1+h_2)$};
\node[left, scale=0.65] at (0,13) {$(0,2h_1+h_2)$};
\filldraw[color=red!70, fill=red!50, thick](0,0) circle (0.2);
\filldraw[color=red!70, fill=red!50, thick](10,0) circle (0.2);
\filldraw[color=red!70, fill=red!50, thick](0,8) circle (0.2);
\filldraw[color=red!70, fill=red!50, thick](0,13) circle (0.2);
\filldraw[color=red!70, fill=red!50, thick](14,0) circle (0.2);
\filldraw[color=red!70, fill=red!50, thick](14,5) circle (0.2);
\filldraw[color=red!70, fill=red!50, thick](4,13) circle (0.2);
\node[right, violet ,scale=0.9] at (8.5,10.5) {$P$};
\end{tikzpicture}
\end{center}
\caption{The fiber polytopes $I_1,I_2,P$.}\label{fig:MFP2}
\end{figure}
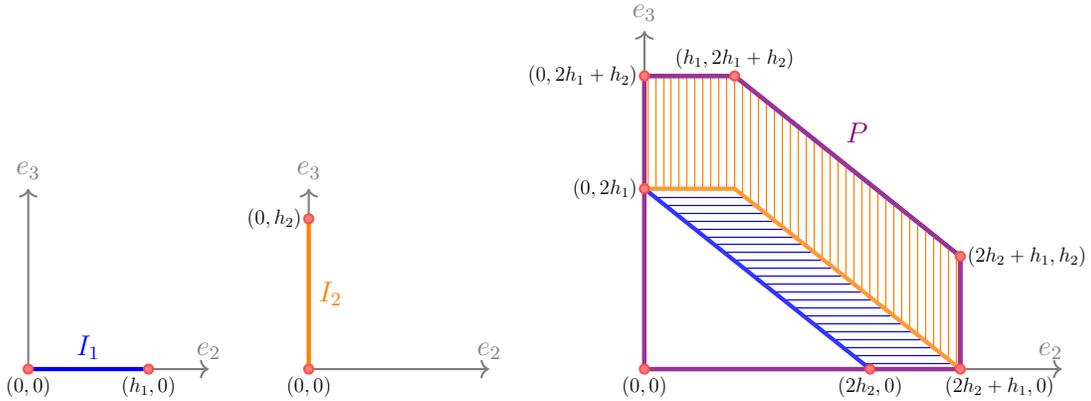
\end{exa}

\subsection{The proof of Lemma \ref{labstr} and Theorem \ref{thsum}}

\begin{proof}[Proof of Lemma \ref{labstr}]
To prove Part 1, we compactify the torus $T\supset C$ with the toric variety $X$ associated to the Newton polygon $N$ of the curve $C$. Then we choose a generic Laurent polynomial $f:T\to\C$ close to the defining equation of $C$ and having the same Newton polygon $N$. 

The closure $\bar Y$ of the curve $Y=\{f=0\}$ is a smooth curve in $X$ approximating the closure $\bar C$. In order to compare the Euler characteristics of $Y$ and $C$, we observe that it is the same in a small metric neighborhood of every point of $\bar C$, except for the following finite sets of points: $S:=\sing C$ and $S_\infty:=\bar C\setminus T$. In a small metric neighborhood $U_s$ of each of the latter points $s$, we have:
$$e(Y\cap U_s)-e(C\cap U_s)=e(\mbox{Milnor fiber of }C\mbox{ at }s)-1=-\mu(s)\mbox{ if } s\in S,$$

\vspace{0.5ex}

\noindent where $\mu(s)$ stands for the Milnor number of $C$ at $s$, and
$$\sum_{s\in (R-\mbox{\scriptsize orbit of }X)}e(Y\cap U_s)-e(C\cap U_s)=e((\mbox{Milnor fiber of } \bar C\mbox{ at }s)\cap T)=-\sum_{i,j}g^R_{i,j}\mbox{ if } s\in S_\infty,$$
where the latter equality follows from Proposition \ref{deltamulti}.

Summing up these equalities, we conclude that $$e(Y)-e(C)=-\sum_{s\in S}\mu(s)-\sum_{R,i,j}g^R_{i,j},$$
where $R$ runs over all exterior normal rays of the Newton polygon $N$. By the BKK formula $e(Y)=-($lattice area of $N)$, we arrive at the first part of Lemma \ref{labstr}.

Then, denoting by $b(s)$ the number of branches of $C$ at $s\in S$, another Euler characteristics count
$$e(C)=e(n(C))-\sum_{s\in \mathcal S}(b(s)-1)$$
and the Milnor-Jung expression $\frac{\mu(s)+b(s)-1}2$ for the $\delta$-nvariant of $C$ at $s$ lead us from Part 1 to Part 2 of Lemma \ref{labstr}.

\end{proof}

\begin{proof}[Proof of Theorem \ref{thsum}]
The statement follows from Part $2$ of Lemma \ref{labstr}. The Newton polygon of $C$ is the fiber polygon of $\conv(A_1),\ldots,\conv(A_{n+1})$ with respect to the coordinate projection $\CC^{n+2}\twoheadrightarrow\CC^2.$ An explicit formula for its area is given in Proposition \ref{fib_area} and equals the first summand of the desired expression. In our setting, the normalization $n(C)$ is the complete intersection curve $V$ together with the intersection points of its compactification $\widebar{V}$ and the orbits of the toric variety $X_{\Omega}$ corresponding to the horizontal facets of the polytope $\conv(A_1)+\ldots+\conv(A_{n+1}).$ The Euler characteristic of $V$ can be computed via the Bernstein-Kouchnirenko-Khovanskii formula, and the result yields the second summand. The contribution of the points added to $V$ is computed via the Bernstein-Kouchnirenko formula and yields the third summand. Finally, by Lemma \ref{projtangency} the the tangency matrices $g^R_{i,j}$ coincide with the combinatorial tangency matrices $g^{\delta}_{i,j}$ defined in subsection \ref{subs:tangency_matrix}, which concludes the proof of the theorem. 
\end{proof} 

\section{Singular loci of resultants, and generalized Vandermonde matrices}\label{sschur}

In this section, we prove a significant part of Theorem \ref{th0res}.

\begin{utver}\label{propresord} 1) In the setting of Theorem \ref{th0res}, the singular locus of the resultant $R_B$ is the union of the irreducible sets $T_0,T_1,T_2,S_0,S_2,S_3,\ldots,S_\infty$ defined (when applicable) by the formulas $(1-4)$ therein, and one more algebraic set of codimension 2, denoted by $S_1$.

2) There is a codimension 3 algebraic subset $\Sigma\subset\C^B$, such that, in a small neighborhood of every point of the strata $\tilde S_i:=S_i\setminus\Sigma,\, i=1,2,3,\ldots$, and $\tilde T_j:=T_j\setminus\Sigma,\, j=0,1,2$, the resultant $R_B$ is the union of several smooth pairwise transversal hypersurfaces, whose number for each stratum is specified in Table 1. (In particular, these strata are smooth and really contained in $\sing R_B$.)
\end{utver}

\subsection{Stratification on the pairs of polynomials}

For the beginning, we will narrow down our study to
$$\C^{B_1\x} \x \C^{B_2\x}=\left\{ (f_1,f_2) \in \C^{B_1} \x \C^{B_2} \,|\, f_1 \not\equiv 0 \text{ and } f_2 \not\equiv 0 \right\},$$
and then consider the subsets $S \in \C^{B_1\x} \x \C^{B_2\x}$ along with their closures $\overline S \in \C^{B_1} \x \C^{B_2}$.

\begin{defin}
Let $f_1(x) \in \C^{B_1\x}$ and $f_2(c) \in \C^{B_2\x}$ be two non-zero (Laurent) polynomials, and $x \in \CP^1$ be a point of the complex line. The multiplicity $\ord_x(f_i)$ of $f_i$ at $x$ is
\begin{enumerate}
    \item The usual multiplicity of a polynomial $f_i$ at $x$, if $x \neq 0, \infty$.
    \item The minimal integer $n \geqslant 0$ such that the $((\min B_i)+n)$-th coefficient of $f_i$ is non-zero, if $x=0$.
    \item The minimal integer $n \geqslant 0$ such that the $((\max B_i)-n)$-th coefficient of $f_i$ is non-zero, if $x=\infty$.
\end{enumerate}
The multiplicity $\ord_x(f_1, f_2)$ of a pair $(f_1, f_2)$ at $x$ is the minimum of $\ord_x(f_1)$ and $\ord_x(f_2)$.
\end{defin}

\begin{defin}\label{defN}
A symmetric filtration subset
$$N_{j_0}^{j_\infty}(j_1,\ldots,j_k),\quad j_0, j_\infty \geqslant 0, \quad j_1,\ldots,j_k \geqslant 1, \quad k \geqslant 0$$
consists of $(f_1,f_2) \in \C^{B_1\x} \times \C^{B_2\x}$ such that
\begin{enumerate}
    \item $\ord_0(f_1, f_2) \geqslant j_0$ and $\ord_\infty(f_1, f_2) \geqslant j_\infty$;
    \item $f_1=f_2=0$ has at least $k$ distinct solutions $x_1,\ldots,x_k$ in $\C^\times$;
    \item at the $m$-th solution $x_m$ holds $\ord_{x_m}(f_1, f_2) \geqslant j_m$.
\end{enumerate}
\end{defin}

\begin{lemma}[see Cor. 2.14 of \cite{S21}]\label{lemcortranstype}
If $\langle B_1,B_2\rangle=1$, then the intersection of the singular locus $\sing R_B$ with $\C^{B_1\x}\x\C^{B_2\x}$ is
$$N(2) \cup N^2_0 \cup N^0_2 \cup N(1,1) \cup N^1_0(1) \cup N^0_1(1).$$
\end{lemma}

\begin{lemma}\label{corcompoutN11}
In the setting of Theorem \ref{th0res}, the irreducible components of the singular locus $\sing R_B$ are following:
\begin{enumerate}
    \item The irreducible components of $\overline{N(1,1)}$;
    \item $S_0$, $S_\infty$, $T_1$ and $T_2$, if applicable.
\end{enumerate}
\end{lemma}

To prove it, we will use the following theorem:

\begin{theor}[see Th. 1.2 (ii) of \cite{S21}]\label{theor1main2}
If $\langle B_1,B_2\rangle=1$ and not $L(B_1)=L(B_2)=1$, then the singular locus $\sing R_B$ consists of several irreducible components of codimension 2.
\end{theor}

\begin{proof}
First, consider the irreducible components of $\sing R_B$ whose intersection with $\C^{B_1\x}\x\C^{B_2\x}$ has codimension 2.
The codimension of (every irreducible component) of $N(2)$, $N^1_0(1)$ and $N^0_1(1)$ is at least 3, so they do not contribute to the irreducible components of $\sing R_B$. Unless $(1+\min B_i)\in B_i$ for $i=1$ or 2, the subset $N^2_0$ has codimension 2, thus $\overline {N^2_0}$ is an irreducible component of $\sing R_B$, namely $S_0$. Similarly, unless $(\max B_i-1)\in B_i$ for $i=1$ or 2, the set $\overline {N^0_2}=S_\infty$ is an irreducible component of $\sing R_B$.

Furthermore, the complement of $\C^{B_1\x}\x\C^{B_2\x}$ consists of two vector spaces of codimensions $|B_1|$ and $|B_2|$ respectively. If $|B_1|=2$, then $\{0\}\x\C^{B_2}$ is an irreducible component of $\sing R_B$, namely $T_1$. Similarly, if $|B_2|=2$, then $\C^{B_1}\x\{0\}=T_2$ is an irreducible component of $\sing R_B$. Otherwise, the codimension of the complement is at least 3, so it does not contribute to the irreducible components of $\sing R_B$.
\end{proof}

\subsection{Solution space and its stratification}

From now on we consider only the case $j_0=j_\infty=0$ and study the filtration subsets only of the form $N(j_1,\ldots,j_k)$. Consider the solution space
$$Z_k=\left\{ (x_1, \ldots, x_k) \subset (\C^\times)^k \,|\, x_i \neq x_j \; \text{for} \; i \neq j \right\}$$
of tuples of different non-zero numbers. We would like to define a stratification on it that is related with the structure of $N(j_1,\ldots,j_k)$.

\begin{defin}\label{defmatrsimple}
First let us for each tuple $x=(x_1,\ldots,x_k)$ define a pair of polynomials
$$h^1_x(t)=\prod_{m=1}^k (t-x_m)^{j_m} \quad \text{ and } \quad h^2_x(t)=\prod_{m=1}^k (t-x_m)^{j_m}.$$

Then consider a pair of maps
$$\psi^1_x: \C^{B_1} \to \C[t]/(h^1_x(t)) \quad \text{ and } \quad \psi^2_x: \C^{B_2} \to \C[t]/(h^2_x(t))$$
given by
$$f_1(t) \mapsto f_1(t) \mod h^1_x(t) \quad \text{ and } \quad f_2(t) \mapsto f_2(t) \mod h^2_x(t).$$

We can now define the strata
$$S^1_{n_1} = \left\{ (x_1, \ldots, x_k) \subset Z_k \,|\, \codim \mathop{\rm Im} \psi^1_x = n_1 \right\},$$
$$S^2_{n_2} = \left\{ (x_1, \ldots, x_k) \subset Z_k \,|\, \codim \mathop{\rm Im} \psi^2_x = n_2 \right\},$$
$$S_{n_1, n_2} = S^1_{n_1} \cap S^2_{n_2} = \left\{ (x_1, \ldots, x_k) \subset Z_k \,|\, \codim \mathop{\rm Im} \psi^1_x = n_1 \text{ and } \codim \mathop{\rm Im} \psi^2_x = n_2 \right\}.$$
\end{defin}

\begin{rem}\label{remM}
The polynomial $h^i_x(t)$ vanishes on $x_m$: $h^i_x(x_m)=0$; moreover, a number of its derivatives also vanish: $(h^i_x)^{(d^i_m)}(x_m)=0$ for $d^i_m<j_m$; as a consequence, the space $\C[t]/(h^i_x(t))$ has a system of coordinate functions sending the class $f_i(t) \mod h^i_x(t)$ to
$$f_i(x_1), f_i'(x_1), \ldots, f_i^{(j_1-1)}(x_1), \quad f_i(x_2), f_i'(x_2), \ldots, f_i^{(j_2-1)}(x_2), \quad \ldots, \quad f_i(x_k), f_i'(x_k), \ldots, f_i^{(j_k-1)}(x_k).$$
These functions form a basis of the dual space. Using this basis, we can write down the matrix $M_i$ of the map $\psi^i_x$ to find out $\codim \mathop{\rm Im} \psi^i_x=\cork M_i$.

For example, suppose that $N(j_1,\ldots,j_k)=N(1,1)$. Then the matrix of the map $\psi^i_x$ has the form
\begin{equation}
M_i=M_i(x,y) = \begin{pmatrix}
x^{b^i_1} & x^{b^i_2} & x^{b^i_3} & \ldots & x^{b^i_{s_i}} \\
y^{b^i_1} & y^{b^i_2} & y^{b^i_3} & \ldots & y^{b^i_{s_i}}
\end{pmatrix}
\end{equation}
where the columns are numbered by the elements of $B_i=\{b^i_1, \ldots, b^i_{s_i}\}$ and $x_1,x_2$ are replaced by $x,y$.

Similarly, for $N(1,1,1)$ the matrix has the form
\begin{equation}\label{eqM}
M_i=M_i(x,y,z) = \begin{pmatrix}
x^{b^i_1} & x^{b^i_2} & x^{b^i_3} & \ldots & x^{b^i_{s_i}} \\
y^{b^i_1} & y^{b^i_2} & y^{b^i_3} & \ldots & y^{b^i_{s_i}} \\
z^{b^i_1} & z^{b^i_2} & z^{b^i_3} & \ldots & z^{b^i_{s_i}}.
\end{pmatrix}
\end{equation}
\end{rem}

\subsection{Strata of the solution space and their W-images}

Now we will prove a lemma that allows us to find the irreducible components of strata $N(j_1,\ldots,j_k)$ from the irreducible components of strata $S_{n_1,n_2}$. It is a generalization of Lemma 4.4 of \cite{S21}. For this, we would need the following definition.

\begin{defin}
Let $S \subset Z_k$ be an algebraic subset and fix $N(j_1,\ldots,j_k)$. Then $W(S)$ is a set of $(f_1,f_2) \in \C^{B_1\x} \x \C^{B_2\x}$ such that there exists $(x_1,\ldots,x_k) \in S$ for which we have $ \ord_{x_m}(f_1) \geqslant j_m$ and $ \ord_{x_m}(f_2) \geqslant j_m$ for all $ m=1,\ldots,k$.
\end{defin}

\begin{lemma}\label{lemWcodim}
If $S \subset S_{n_1,n_2}$ is an algebraic subset of codimension $n_1+n_2+c$, then $W(S)$ is an algebraic subset of codimension $\sum_{m=1}^k(2j_m-1)+c$ (maybe empty).

Moreover, if $S$ is irreducible, then $W(S)$ is irreducible or empty, and it is empty if and only if $|B_1| \leqslant \sum_{m=1}^{k} j_m-n_1$ or $|B_2| \leqslant \sum_{m=1}^k j_m-n_2$.
\end{lemma}

\begin{proof}
Let us define a subset $\widetilde N=\widetilde N (j_1,\ldots,j_k)$ consisting of $(f_1;\ f_2;\ x_1, \ldots, x_k) \in \C^{B_1\x} \times \C^{B_2\x} \times Z_k$ such that $\ord_{x_m}(f_1) \geqslant j_m$ and $\ord_{x_m}(f_2) \geqslant j_m$ for $m=1,\ldots,k$, and also let us consider two projections:
$$p: \C^{B_1\x} \times \C^{B_2\x} \times Z_k \to \C^{B_1\x} \times \C^{B_2\x} \quad \text{ and } \quad q: \C^{B_1\x} \times \C^{B_2\x} \times Z_k \to Z_k.$$
In this terms, $W(S)=p(q^{-1}(S) \cap \widetilde N)$.

The fiber $q^{-1}(x) \cap \widetilde N$ over the point $x=(x_1, \ldots, x_k) \in Z_k$ is a subset in $\C^{B_1\x} \times \C^{B_2\x}$ which is the direct product of the spaces
$$\ker \psi^1_x \setminus \{0\} = \left\{ f_1 \in \C^{B_1\x} \,|\, \ord_{x_m} f_1 \geqslant j_m \right\} \text{ and } \ker \psi^2_x \setminus \{0\}  = \left\{ f_2 \in \C^{B_2\x} \,|\, \ord_{x_m} f_2 \geqslant j_m \right\}.$$

Let us notice that
$$\codim \ker \psi^i_x=\dim \mathop{\rm Im} \psi^i_x=\sum_{m=1}^{k} j_m-\codim \mathop{\rm Im} \psi^i_x=\sum_{m=1}^{k} j_m-n_i,$$
and, moreover, that either $\codim (\ker \psi^i_x \setminus \{0\}) = \codim \ker \psi^i_x$, or $\ker \psi^i_x \setminus \{0\}$ is empty. Moreover, it is empty if and only if $\sum_{m=1}^{k} j_m-n_i$ is equal to the dimension of the enclosing space, that is to $|B_i|$. As the dimension can not be negative, it is equal to the condition $|B_i| \leqslant \sum_{m=1}^{k} j_m-n_i$.

Recall that the fiber is $q^{-1}(x) \cap \widetilde N=\ker \psi^1_x \setminus \{0\} \x \ker \psi^2_x \setminus \{0\}$, thus its codimension is $2\sum_{m=1}^{k} j_m - n_1 - n_2$ (or its is empty, which happens if and only if $|B_1| \leqslant \sum_{m=1}^{k} j_m-n_1$ or $|B_2| \leqslant \sum_{m=1}^k j_m-n_2$).

Consider the projection $q^{-1}(S) \cap \widetilde N \to S$. As $S \subset S_{n_1,n_2}$, its fibers $q^{-1}(x) \cap \widetilde N$ are punctured vector spaces of the same dimension (or are all empty), so it is a fiber bundle (or is empty). If the base space $S$ is irreducible, then $q^{-1}(S) \cap \widetilde N$ also is irreducible (or empty).

The codimension of $q^{-1}(S) \cap \widetilde N$ inside $q^{-1}(S)$ is equal to the codimension of $q^{-1}(x)$ inside $\C^{B_1\x}\x\C^{B_2\x}\x\{x\}$, that is to $2\sum_{m=1}^{k} j_m - n_1 - n_2$ (or $q^{-1}(S)$ is empty). The codimension of $S$ is $n_1+n_2+c$, thus the codimension of $q^{-1}(S) \cap \widetilde N$ inside $\C^{B_1\x} \times \C^{B_2\x} \times Z_k$ is $2\sum_{m=1}^{k} j_m+c$ (or $q^{-1}(S)$ is empty).

The image of $\widetilde N$ under the projection $p: \C^{B_1\x} \times \C^{B_2\x} \times Z_k \to \C^{B_1\x} \times \C^{B_2\x}$ is exactly $N=N(j_1,\ldots,j_k)$, and the preimage of any point of $N$ consists of a finite number of points. Thus $\dim W(S)=\dim p(q^{-1}(S) \cap \widetilde N)=\dim q^{-1}(S) \cap \widetilde N$. If $S$ is irreducible, then $W(S)$ also is irreducible (or empty). Moreover, $\dim \C^{B_1\x} \times \C^{B_2\x} \times Z_k=\dim \C^{B_1\x} \times \C^{B_2\x}+k$, thus $\codim W(S)=\codim q^{-1}(S) \cap \widetilde N-k=\sum_{m=1}^k(2j_m-1)+c$ (or $W(S)$ is empty).

We can easily see that the sets above are empty if and only if $|B_1| \leqslant \sum_{m=1}^k j_m-n_1$ or $|B_2| \leqslant \sum_{m=1}^k j_m-n_2$.
\end{proof}

\begin{rem}
If each $S_{n_1,n_2}=\bigsqcup_i S_{n_1,n_2}^i$ is written as a union $S_{n_1,n_2}=\bigcup_i S_{n_1,n_2}^i$ of irreducible algebraic subsets $S_{n_1,n_2}^i$, then the union of the corresponding subsets $W(S)$ is the whole $N(j_1,\ldots,j_k)$ (but this union is not in general disjoint).
\end{rem}

\subsection{Components of N(1,1)}

Now let us find out what $S_{n_1,n_2} \subset Z_2$ are in the case of $N(1,1)$.

Let $k_i=\langle B_i\rangle$, so that $B_i$ can be shifted to the proper sublattice $k_i\Z$, but not to any larger sublattice.

Let $\epsilon^1_{k_i}, \ldots, \epsilon^{k_i-1}_{k_i}$ be $k_i$-th roots of unity different from 1, and let $E^1_{k_i}, \ldots, E^{k_i-1}_{k_i}$ be the following subsets of $Z_2$: 
$$E^j_{k_i}=\left\{ (x_1, x_2) \in Z_2 \,|\, x_1/x_2=\epsilon^j_{k_i} \right\}.$$

\begin{lemma}[see Lem. 4.8, 4.11 of \cite{S21}]
The set $Z_2$ has the following decompositions:
\begin{enumerate}
    \item $Z_2=S^i_0 \sqcup S^i_1$, where $S^i_1=E^1_{k_i} \sqcup \ldots \sqcup E^{k_i-1}_{k_i}$.
    \item If $\langle B_1,B_2\rangle=1$ (in other words, if $k_1$ and $k_2$ are coprime), then $Z_2=S_{00} \sqcup S_{10} \sqcup S_{01}$, where $S_{10}=E^1_{k_1} \sqcup \ldots \sqcup E^{k_1-1}_{k_1}$ and $S_{01}=E^1_{k_2} \sqcup \ldots \sqcup E^{k_2-1}_{k_2}$.
    \end{enumerate}
\end{lemma}

\begin{sledst}\label{corN11}
From Lemma \ref{lemWcodim} and the remark after it, we deduce the following decomposition:
$$N(1,1)=WS_{00} \cup \left( WE^1_{k_1} \cup \ldots \cup WE^{k_1-1}_{k_1} \right) \cup \left( WE^1_{k_2} \cup \ldots \cup WE^{k_2-1}_{k_2} \right).$$

If $|B_1|>2$ and $|B_2|>2$, then $WS_{00}$ and $WE^j_{k_i}$ are irreducible subsets of codimension 2 (not necessarily different from each other). If $|B_1|=2$, then $WE^j_{k_2}$ are empty, and similarly in the case $|B_2|>2$.
\end{sledst}

Here we can explicitly describe the components of $N(1,1)$ as follows:
$$WE^j_{k_i}=
\left\{ (f,g) \in \C^{B_1\x} \x \C^{B_2\x} \,|\, \exists x \in \C^\x: f(x)=g(x)=f\left(x\cdot\epsilon^j_{k_i}\right)=g\left(x\cdot \epsilon^j_{k_i}\right)=0\right\},$$
thus, in terms of Theorem \ref{th0res}, its closure is $S_i^j$, while the closure of $WS_{00}$ is $S$.

The problem that we solve in the rest of the section is that these components partially coincide.

\subsection{Schur polynomials}

\begin{defin}
Consider a triple of integers $0 \leqslant a<b<c$. It defines a polynomial
$$\det\nolimits_{a,b,c}(x,y,z)=\begin{pmatrix}
x^a & x^b & x^c\\
y^a & y^b & y^c\\
z^a & z^b & z^c
\end{pmatrix}.$$
\end{defin}

This polynomial is divisible by the Vandermonde determinant $\det\nolimits_{0,1,2}(x,y,z)=(x-y)(y-z)(z-x)$. The quotient $\det\nolimits_{a,b,c}/\det\nolimits_{0,1,2}$ is called a Schur polynomial. We would like Schur polynomials to be irreducible, but it is not in general true; however, the following theorem of Dvornicich and Zannier is true:
\begin{theor}[see Th.3.1. of \cite{zannier}]
Suppose $a=0$ and denote $GCD(b,c)$ by $n$. The polynomials $\det\nolimits_{0,b,c}/\det\nolimits_{0,n,2n}$ are irreducible, except for the cases of the form $b=n$ and $c=2n$, in which they are constant.
\end{theor}

\begin{sledst}\label{corGCDdet}
Consider two triples of integers $0=a_1<b_1<c_1$ and $0=a_2<b_2<c_2$. Denote $GCD(b_i,c_i)$ by $n_i$ and $GCD(n_1,n_2)$ by $n$. Then
$$GCD(\det\nolimits_{0,b_1,c_1}, \det\nolimits_{0,b_2,c_2})=\det\nolimits_{0,n,2n},$$
unless $b_1=b_2$ and $c_1=c_2$.
\end{sledst}

\begin{proof}
Let us denote the irreducible polynomials $\det\nolimits_{0,b,c}/\det\nolimits_{0,n,2n}$ from the theorem above by $T_{0,b,c}$. Then $\det\nolimits_{0,b_i,c_i}=T_{0,b_i,c_i} \cdot \det\nolimits_{0,n_i,2n_i}$.

Let us lexicographically order the monomials $x^jy^kz^l$. The highest monomial of $\det\nolimits_{0,b_i,c_i}$ is $x^{c_i}y^{b_i}$, thus the highest monomial of $T_{0,b_i,c_i}$ is $x^{c_i-2n_i}y^{b_i-n_i}$. Thus either $b_1=b_2$ and $c_1=c_2$, or the highest monomials of the irreducible polynomials $T_{0,b_1,c_1}$ and $T_{0,b_2,c_2}$ do not coincide, thus these polynomials are different.

The polynomials $T_{0,b_1,c_1}$ are irreducible and symmetric in $a,b,c$, while the polynomials $x^{n_i}-y^{n_i}$, $y^{n_i}-z^{n_i}$ and $z^{n_i}-x^{n_i}$ (whose product is $\det\nolimits_{0,n_i,2n_i}$) depend only on two of the variables $x$, $y$ and $z$, thus they can not have a non-constant common divisor.

Consequently,
$$GCD(\det\nolimits_{0,b_1,c_1}, \det\nolimits_{0,b_2,c_2})=GCD(\det\nolimits_{0,n_1,2n_1}, \det\nolimits_{0,n_2,2n_2})=\det\nolimits_{0,n,2n}.$$
\end{proof}

\subsection{Vandermonde matrices}

In this subsection, we work over an arbitrary field $k$ of characteristics 0.

\begin{lemma}\label{lem3minor}
Let $x$, $y$ and $z$ be roots of unity, let $a$, $b$ and $c$ be integer numbers and consider a matrix
$$M(a,b,c;x,y,z) = \begin{pmatrix}
x^a & x^b & x^c \\
y^a & y^b & y^c \\ 
z^a & z^b & z^c
\end{pmatrix}$$
If this matrix is degenerate, then it has either two proportional rows or two proportional columns.
\end{lemma}

\begin{proof}
Without loss of generality we can divide the columns by $z^a$, $z^b$ and $z^c$ and thus replace $M(a,b,c;x,y,z)$ with $M(a,b,c;x/z,y/z,1)$. We denote this $x/z$ and $y/z$ by $x$ and $y$ and consider
$$M=M(a,b,c;x,y,1) = \begin{pmatrix}
x^a & x^b & x^c \\
y^a & y^b & y^c \\
1 & 1 & 1
\end{pmatrix}$$
Let us write the equation $\det M=0$ as
$$(x^a-x^b)(y^a-y^c)=(x^a-x^c)(y^a-y^b).$$
Suppose that $x^a=x^b$. Then either $x^a=x^c$ and $M$ has two proportional rows or $y^a=y^b$ and $M$ has two proportional (even equal) columns.

As the conditions of the lemma are symmetric in $a$, $b$ and $c$, and are symmetric in $x$ and $y$, we can now suppose that the numbers $x^a$, $x^b$, $x^c$ are pairwise different and that the numbers $y^a$, $y^b$, $y^c$ are also pairwise different. Thus we can write the condition above as
$$\frac{x^a-x^b}{x^a-x^c}=\frac{y^a-y^b}{y^a-y^c} \neq 0.$$

Suppose that $k=\C$ and consider the complex conjugation. We would like to prove that
$$\left(\frac{x^a-x^b}{x^a-x^c}\right) / \overline{\left(\frac{x^a-x^b}{x^a-x^c}\right)} = \frac{x^b}{x^c}.$$

Indeed, one just opens up the brackets and uses the equalities $x^a\overline x^a=x^b\overline x^b=x^c\overline x^c=1$ to prove it. Consequently, $x^b/x^c=y^b/y^c$, so $M$ has two proportional columns.

Now let $k$ be any field of characteristics $0$. Without loss of generality we can suppose that it is algebraically closed. If $x$ and $y$ are two roots of unity (maybe equal), then they generate a finite group, but a finite subgroup of the multiplicative group of a field is cyclic. Let us denote its generator by $\epsilon$, take an automorphism $\phi$ of $\Q(\epsilon)$ such that $\phi(\epsilon)=\epsilon^{-1}$ and extend it to an  automorphism $\overline\phi$ of $k$. Then $\overline\phi$ is a field automorphism over $\Q$ such that $\overline\phi(x)=x^{-1}$ and $\overline\phi(y)=y^{-1}$, that is $x\overline\phi(x)=y\overline\phi(y)=1$, and we can repeat the calculation above using $\overline\phi$ instead of the complex conjugation to prove that $x^b/x^c=y^b/y^c$.
\end{proof}

\begin{defin}
Let $M$ be a matrix with non-zero entries. The combinatorial rank $\crk M$ is the minimum of the number of classes of proportional rows of $M$ and the number of classes of proportional columns of $M$.
\end{defin}

In these terms, Lemma \ref{lem3minor} says that for a certain class of $3 \x 3$-matrices, if $\rk M < 3$, then $\crk M < 3$.

\begin{lemma}\label{lem3prop}
Consider an $m \x n$-matrix $M$ with non-zero entries such that $m < n$. If for every maximal minor $N$ the combinatorial rank $\crk N < m$, then $\crk M < m$.
\end{lemma}

\begin{proof}
Consider the columns as the points of $\mathbb P^{m-1}$. The condition that the combinatorial rank of $M$ (resp., a maximal minor $N$) is less than $m$ means that
\begin{enumerate}
    \item either the columns of $M$ (resp., $N$) form less than $m$ points of $\mathbb P^{m-1}$,
    \item or all of these points lie in the same hypersurface of the form
$$H^{i,j}_k=\{(x_1:\ldots:x_m) | x_i/x_j=k \}.$$
\end{enumerate}

If the number of points of $P^{m-1}$ formed by the columns of $M$ is less than $m$, the theorem is proven. Otherwise let $L$ be the projective hull of these points. It is a projective space of dimension at most $(m-1)$, thus we can take $m$ points that generate it. By the condition on the maximal minor $N$ containing all the columns defining these $m$ points, all of them lie in a certain hypersurface $H^{i,j}_k$, thus $L \subset H^{i,j}_k$ and all of the points defined by the columns of $M$ lie in $H^{i,j}_k$, that is the combinatorial rank $\crk M < m$.
\end{proof}

If $B_1, \ldots, B_n \subset Z$ are finite sets of integers, then $\langle B_1,\ldots,B_n\rangle$ is the maximal integer $b$ such that for every $j$ and every $b',b'' \in B_j$ the number $b'-b''$ is divisible by $b$. In particular, $\langle B_1,\ldots,B_n\rangle=GCD \langle B_j \rangle$.

\begin{defin}
Let $X_1,\ldots,X_m$ be finite sets of roots of unity. Define $\langle X_1,\ldots,X_n \rangle$ to be the minimal integer $a$ such that for every $i$ and every $x', x'' \in X_i$ the number $x'/x''$ is $a$-th root of unity. In particular, $\langle X_1,\ldots,X_m \rangle=LCM \langle X_i \rangle$.
\end{defin}

\begin{rem}\label{remconj1}
Let $X=\{ x_1, \ldots, x_m \}$ be a finite set of roots of unity, let $B=\{ b_1, \ldots, b_n \}$ be a finite set of integer numbers, and consider an $m \times n$-matrix $M=(x_i^{b_j})$. If its rank is 1, then $\langle X \rangle$ is divisible by $\langle B \rangle$.
\end{rem}

\begin{lemma}\label{lemMsplits2}
Let $X=\{ x_1, x_2, x_3 \}$ be a finite set of roots of unity, let $B=\{ b_1, \ldots, b_n \}$ be a finite set of integer numbers, and consider an $3 \times n$-matrix $M=(x_i^{b_j})$. If its rank is 2, then either $X=X_1 \sqcup X_2$ such that $\langle X_1, X_2 \rangle$ is divisible by $\langle B \rangle$, or $B=B_1 \sqcup B_2$ such that $\langle X \rangle$ is divisible by $\langle B_1, B_2 \rangle$.
\end{lemma}

\begin{proof}
The rank of $M$ is less then 3, thus every its $3 \x 3$-minor vanishes. By Lemma \ref{lem3minor}, every its such minor has either two proportional rows or two proportional columns. By Lemma \ref{lem3prop}, either the rows or the columns of $M$ can be divided into two groups of mutually proportional elements. If $X=X_1 \sqcup X_2$, such that the rows from the same $X_i$ are proportional, then we have two matrices $M_1$ and $M_2$, consisting from the rows of $M$ from $X_1$ and $X_2$ respectively. Both $M_i$ have rank 1, thus by Remark \ref{remconj1}, $\langle X_i \rangle$ is divisible by $\langle B \rangle$, thus $\langle X_1, X_2 \rangle$ is also divisible by $\langle B \rangle$. The case of $B=B_1 \sqcup B_2$ is similar.
\end{proof}

\subsection{Components of N(1,1,1)}

\begin{lemma}\label{lemcompN111}
Unless $|B_1|=3$ and $B_1=B_2 + const$, the closures of the irreducible components of $N(1,1,1)$ of codimension 2 are $S_m$ for $m \geqslant 3$ as defined in Theorem \ref{th0res}.
\end{lemma}

If $|B_1|=3$ and $B_1=B_2 + c$, then there is a codimension 2 subset $\{(f,g) \in \C^{B_1}\x\C^{B_2}: f \sim g \cdot x^c\}$, for which there is $L(B_1)$ common roots; if $L(B_1) \geqslant 3$, then it gives us a component of $N(1,1,1)$.

\begin{proof}
According to Lemma \ref{lemWcodim}, if $S \subset S_{n_1,n_2}$ is an irreducible algebraic subset of codimension $n_1+n_2+c$, then either $W(S)$ is an irreducible algebraic subset of codimension $3+c$, or $W(S)$ is empty; moreover, it is empty if and only if $|B_1| \leqslant 3-n_1$ or $|B_2| \leqslant 3-n_2$.

Suppose that $3+c=2$ and check whether $S_{n_1,n_2}$ has an irreducible component $S$ of codimension $n_1+n_2-1$ and, if it has one, what are $W(S)$ and its closure $\overline{W(S)}$. Here we use the matrix definition of $S_{n_1,n_2}$, see Remark \ref{remM}.

$S_{0,0}$ obviously can not have an irreducible component of codimension $-1$.

$S_{1,0}$: If $|B_1|=2$, then $|B_1| \leqslant 3-n_1=2$, thus any $W(S)$ would be empty anyway; otherwise, $\widehat S^1_1=S^1_1 \cup S^1_2$ is a proper closed subset defined by vanishing of non-zero number of polynomials, thus $S_{1,0}=S^1_1 \cap S^2_0$ can not have an irreducible component of codimension $0$. The same is true for $S_{0,1}$.

$S_{2,0}$: According to Corollary 4.7 of \cite{S21}, $S^i_2$ has codimension at least $2$, thus neither $S_{2,0}$, nor $S_{0,2}$ can not have an irreducible component of codimension $1$. 

$S_{1,1}$: If $|B_i|=2$, then $|B_i| \leqslant 3-n_i=2$, thus any $W(S)$ would be empty anyway; otherwise, according to Corollary 4.7 of \cite{S21}, $S^i_1$ has codimension at least $1$; namely, it lies in the subset defined by the determinants of all 3 × 3-minors $\det_{a,b,c}$ of $M_i(x,y,z)$. Without loss of generality we can suppose that $\min B_1=\min B_2=0$, thus among $\det_{a,b,c}$ are those with $a=0$, as in Corollary \ref{corGCDdet}. Consider GCD of all $\det_{0,b,c}$ for $b$ and $c$ either both from $B_1$, or both from $B_2$. Either $|B_1|=3$ and $B_1=B_2$, or we can find two different polynomials $\det_{0,b,c}$ and deduce from Corollary \ref{corGCDdet} that GCD of all these $\det_{0,b,c}$'s is $\det_{0,n,2n}$, where $n=\langle B_1,B_2\rangle=1$. Consequently, this GCD does not vanish in $Z_3$, thus $S_{1,1}=S^1_1 \cap S^2_1$ can not have an irreducible component of codimension 1.

$S_{1,2}$: If $|B_1|=2$, then $|B_1| \leqslant 3-n_1=2$, thus any $W(S)$ would be empty anyway. Otherwise, if $S$ is an irreducible component of $S_{1,2}$ of codimension 2, then it is a one-dimensional set; from the matrix definition of the strata (see Remark \ref{remM}) we can see that $(x,y,z) \in S_{1,2}$ if and only if $(1,t,u)=(1,y/x,z/x) \in S_{1,2}$, thus $S=\{(c,ct,cu) | c \in \C^\x\}$ for some $(1,t,u) \in S_{1,2}=S^1_1 \cap S^2_2$. According to Lemma 4.6 of \cite{S21}, if $(1,t,u) \in S^2_2$, then the numbers $t$ and $u$ are $k_2$-th roots of 1 and $1 \neq t \neq u \neq 1$. According to Lemma \ref{lemMsplits2} with $X=\{1,t,u\}$, $B=B_1$, one of the following case holds:
\begin{itemize}
    \item $X=X_1 \sqcup X_2$ such that $\langle X_1, X_2 \rangle$ is divisible by $k_1$; if $X_1=\{1,t\}$, then $\langle X_1 \rangle$ is divisible by $k_1$, thus $t$ is $k_1$-th root of unity, but it is already $k_2$-th root of unity, thus $t=1$, which gives us a contradiction; the case of $X_1=\{1,u\}$ and $X_1=\{t,u\}$ give similar contradictions.
    \item $B_1=B_{1,1}\sqcup B_{1,2}$ such that $\langle X \rangle$ is divisible by $\langle B_{1,1}, B_{1,2} \rangle$; it means that $t$ and $u$ are $m$-th roots of unity for $\langle B_{1,1},B_{1,2},B_2\rangle=m$; but $1\neq t\neq u\neq 1$, thus $m \geqslant 3$.
\end{itemize}
In particular, as $|B_1| > 2$, there exists $S_m$ for $m \geqslant 3$ as defined in Theorem \ref{th0res}. If $(f_{1,1}, f_{1,2}, f_2) \in S_m \cap (\C^{B_1\x}\x\C^{B_2\x})$ have a common root $c$, then $ct$ and $cu$ are also its common roots, thus $(f_{1,1}, f_{1,2}, f_2) \in W(S)$. According to Lemma \ref{lemWcodim}, $W(S)$ is an irreducible component of $N(1,1,1)$ of codimension 2. By definition, $S_m$ is an irreducible component of $\sing R_B$ of codimension 2. Consequently, $S_m \cap (\C^{B_1\x}\x\C^{B_2\x}) \subset W(S)$ and $\overline{W(S)}=S_m$. The same is true for $S_{2,1}$.

$S_{2,2}$: According to Lemma 4.9 of \cite{S21}, if $\langle B_1,B_2\rangle=1$, then $S_{2,2}$ is empty.
\end{proof}

\subsection{Proof of the part 1 of Proposition \ref{propresord}}

According to Lemma \ref{corcompoutN11}, to find the irreducible components of $\sing R_B$ we need to find the irreducible components of $\overline{N(1,1)}$. Corollary \ref{corN11} gives the following decomposition:
$$\overline{N(1,1)}=\overline{WS_{00}} \cup \bigcup_{j=1}^{k_1-1} \overline{WE^j_{k_1}} \cup \bigcup_{j=1}^{k_2-1} \overline{WE^j_{k_2}}.$$
where $\overline{WS_{00}}$ and $\overline{WE^j_{k_i}}$ are irreducible subsets of codimension 2, which are not necessarily different from each other, and if $|B_1|=2$, then there is no $\overline{WE^j_{k_2}}$, and similarly in the case $|B_2|=2$.

\begin{defin}
1) The sets $\overline{WE^j_{k_i}}$ and $\overline{WE^{k_i-j}_{k_i}}$ are called opposite.

2) Suppose that there exists $S_m$ for $m \geqslant 2$ obtained from the decomposition $B_1=B_{1,1} \sqcup B_{1,2}$ such that $\langle B_{1,1},B_{1,2},B_2\rangle=m$ and $|B_1|>2$ as defined in Theorem \ref{th0res}. The $(m-1)$ sets 
$$\overline{WE^{k_2/m}_{k_2}}, \ldots, \overline{WE^{jk_2/m}_{k_2}}, \ldots, \overline{WE^{(m-1)k_2/m}_{k_2}}.$$
are called $m$-good (and similarly for $B_2$). The sets that are not $m$-good for any $m$ are called bad.
\end{defin}

Notice that the decomposition of $B_i$ for $i=1$ gives $m$-good subsets $\overline{WE^j_{k_i}}$ for $i=2$.

\begin{lemma}\label{lemSmWE}
1) Opposite sets coincide.

2) All the $m$-good sets coincide with $S_m$.
\end{lemma}

\begin{proof}
1) If the quotient $x_1/x_2$ of two common roots is $\epsilon^j_{k_i}$, then the quotient $x_2/x_1=\epsilon^{k_i-j}_{k_i}$, thus $WE^j_{k_i}=WE^{k_i-j}_{k_i}$.

2) Any $(f,g) \in S_m$ has the following property: if $f(x)=g(x)=0$, then for any $m$-th roots of unity $\epsilon^j_m$ holds $f(x\cdot\epsilon^j_m)=g(x\cdot\epsilon^j_m)=0$. Moreover, for general $(f,g) \in S_m$ there is a non-zero $x$ such that $f(x)=g(x)=0$. If we define
$$WE^j_m=
\left\{ (f,g) \in \C^{B_1\x} \x \C^{B_2\x} \,|\, \exists x \in \C^\x: f(x)=g(x)=f\left(x\cdot\epsilon^j_m\right)=g\left(x\cdot \epsilon^j_m\right)=0\right\},$$
then the general $(f,g)$ lies in $WE^j_m$ for any $j=1,\ldots,m-1$.

Moreover, if $S_m$ is obtained from the decomposition $B_1=B_{1,1} \sqcup B_{1,2}$, then by the definition holds $\langle B_{1,1},B_{1,2},B_2\rangle=m$, thus $k_2$ is divisible by $m$. In general, $k_i$ is divisible by $m$, thus we can rewrite $\epsilon^j_m$ as $\epsilon^{jk_i/m}_{k_i}$ and $WE^j_m$ as $WE^{jk_i/m}_{k_i}$, which are $m$-good sets.

The set $S_m$ and the closures $\overline{WE^{jk_i/m}_{k_i}}$ are all irreducible components of $\sing R_B$ of codimension 2, the dense subset of $S_m$ lies inside $WE^{jk_i/m}_{k_i}$'s, thus
$$S_m=\overline{WE^{k_i/m}_{k_i}}=\ldots=\overline{WE^{jk_i/m}_{k_i}}=\ldots=\overline{WE^{(m-1)k_i/m}_{k_i}}.$$
\end{proof}

Let us now rewrite the decomposition as
$$\overline{N(1,1)}=\overline{WS_{00}} \cup \bigcup_{\text{bad}} \overline{WE^j_{k_1}} \cup \bigcup_{\text{bad}} \overline{WE^j_{k_2}} \cup \bigcup S_m.$$

\begin{lemma}\label{lemcompN11}
1) A general $(f,g)\in S_m$ for $m\geqslant 2$ has exactly $m$ common roots.

2) If $\overline{WE^j_{k_i}}=S_m$, then $\overline{WE^j_{k_i}}$ is $m$-good.

3) For bad $\overline{WE^j_{k_i}}$, a general $(f,g)\in \overline{WE^j_{k_i}}$ has exactly two common roots.

4) If two bad $\overline{WE^j_{k_i}}$ coincide, then they are opposite.

5) If $|B_1|=3$ and $B_1=B_2 + const$, a general $(f,g)\in \overline{WS_{00}}$ has exactly $L(B_1)$ common roots; otherwise, a general $(f,g)\in \overline{WS_{00}}$ has exactly two common roots.

6) The set $\overline{WS_{00}}$ can not coincide neither with $\overline{WE^j_{k_i}}$, not with $S_m$.
\end{lemma}

In particular, we can calculate the number of irreducible components: if both $|B_i|>2$, then it is the number of positive $j\leqslant k_i/2$ that is not divisible by an integer of the form $k_i/m$ for $m$ from some $S_m$, plus the number of $S_m$ that are correctly defined, plus a number between 1 and 5 for $WS_{00}$, $S_0$, $S_\infty$, $T_1$ and $T_2$, depending on which of them are correctly defined.

\begin{proof}
1) If $S_m$ is obtained from the decomposition $B_1=B_{1,1} \sqcup B_{1,2}$, then by the definition $\langle B_{1,1}, B_{1,2}, B_2\rangle=m$. By the properties of the resultant $R_{B_{1,1},B_{1,2},B_2}$, for a general $(f_{1,1}, f_{1,2}, f_2) \in \C^{B_{1,1}}\x\C^{B_{1,2}}\x\C^{B_2}$ the polynomials $f_{1,1}, f_{1,2}$ and $f_2$ have exactly $m$ common roots.

Let us prove that for a general $(f_{1,1}, f_{1,2}, f_2) \in \C^{B_{1,1}}\x\C^{B_{1,2}}\x\C^{B_2}$, all the common roots of $f_{1,1}+f_{1,2}$ and $f_2$ are also the common roots of $f_{1,1}, f_{1,2}$ and $f_2$. Indeed, consider the common roots of $f_{1,1}+tf_{1,2}$ and $f_2$: if the roots do not depend on $t$, then they are common roots of $f_{1,1}, f_{1,2}$ and $f_2$; otherwise the roots do depend on $t$, thus except for a finite number of values $t$ the polynomials $f_{1,1}+tf_{1,2}$ and $f_2$ do not have other common roots then the common roots of $f_{1,1}, f_{1,2}$ and $f_2$.

2) Suppose that $WE^j_{k_i}$ is not $m$-good: for a general $(f,g) \in S_m$ there are common roots $x, x\cdot\epsilon^1_m, \ldots, x\cdot\epsilon^{m-1}_m$ and there are also common roots $y, y\cdot\epsilon^j_{k_i}$, totally at least $(m+1)$ common root, which gives us a contradiction to the previously proven.

3) Suppose that a general $(f,g) \in WE^j_{k_i}$ has more than 2 common roots, thus $(f,g)$ lies in $N(1,1,1)$ and $\overline{WE^j_{k_i}}$ is equal to the closure of an irreducible component of $N(1,1,1)$ of codimension 2. If $|B_1|=3$ and $B_1=B_2 + const$, then $k_1=k_2=1$, thus there is no WE's. Otherwise by Lemma \ref{lemcompN111} we have that $\overline{WE^j_{k_i}}$ is equal to $S_m$ for some $m \geqslant 2$, which gives us a contradiction to the previously proven.

4) Suppose that $\overline{WE^{j_1}_{m_1}}=\overline{WE^{j_2}_{m_2}}$, where $m_i \in \{k_1,k_2\}$ and $j_i \in \{1, \ldots, m_i-1\}$. Then for a general $(f,g)$ from this subset there are $x_1$ and $x_2 \in \C^\x$ such that the numbers $x_1,\  x_1\cdot\epsilon^{j_1}_{m_1},\  x_2$ and $x_2\cdot\epsilon^{j_2}_{m_2}$ are common roots of $f$ and $g$. Let us notice that $m_1$ and $m_2$ are either equal, or coprime. If $\overline{WE^{j_1}_{m_1}}$ and $\overline{WE^{j_2}_{m_2}}$ are neither identical, not opposite, then $(f,g)$ has at least three non-zero common roots, thus lies in $N(1,1,1)$. If $|B_1|=3$ and $B_1=B_2 + const$, then there is no WE's; otherwise $\overline{WE^{j_1}_{m_1}}=\overline{WE^{j_2}_{m_2}}$ is equal to $S_m$ for some $m \geqslant 2$, which gives us a contradiction.

5, 6) We can prove it as in the cases of $WE^j_{k_i}$.
\end{proof}

\begin{utver}\label{propdoublestrara}
In the setting of Theorem \ref{th0res}, the irreducible components of the singular locus $\sing R_B$ are $T_0,T_1,T_2,S_0,S_m,S_\infty,S,S^i_j$ defined (when applicable) therein, and they are pairwise different.
\end{utver}

\begin{proof}
According to Lemma \ref{corcompoutN11}, the irreducible components of the singular locus $\sing R_B$ are the irreducible components of $\overline{N(1,1)}$ and the components $S_0$, $S_\infty$, $T_1$ and $T_2$, if applicable.

According to Corollary \ref{corN11}, there is decomposition $$\overline{N(1,1)}=\overline{WS_{00}} \cup \bigcup_{j=1}^{k_1-1} \overline{WE^j_{k_1}} \cup \bigcup_{j=1}^{k_2-1} \overline{WE^j_{k_2}}.$$
where $\overline{WS_{00}}$ and $\overline{WE^j_{k_i}}$ are irreducible subsets of codimension 2, which are not necessarily different from each other, and some of them are absent if $|B_i|=2$.

According to Lemma \ref{lemSmWE}, opposite $WE$'s coincide, while $m$-good $WE$'s coincide with $S_m$. According to Lemma \ref{lemcompN11}, there are no more coincidences between the components of $\overline{N(1,1)}$. Thus $\overline{WE^j_{k_i}}$ are exactly $S^i_j$ from the part 6 of Theorem \ref{th0res}, while $\overline{WS_{00}}$ is either $T_0$ from the part 5 (if $|B_1|=3$ and $B_1=B_2+const$) or $S$ from the part 6 (otherwise), and all these components are different.

Moreover, $T_1$ and $T_2$ are lying outside $\C^{B_1\x}\x\C^{B_2\x}$, thus they can not coincide with the components of $\overline{N(1,1)}$.

The intersection $S_0 \cap R_B$ is $R_{B'}$ for $B'_i=B_i\setminus\{\min B_i\}$, thus its general element can not have more common roots and $S_0$ can not coincide with the components of $\overline{N(1,1)}$, and similarly with $S_\infty$.

Consequently, all the components written in the statement of this proposition are different.
\end{proof}

Let us denote the union of $S=\overline{WS_{00}}$ (unless $|B_1|=3$ and $B_1=B_2+const$) and of all $S^i_j=\overline{WE^j_{k_i}}$ by $S_1$. So now we have the decomposition that is required. QED.

\subsection{Proof of the part 2 of Proposition \ref{propresord}}

\begin{defin}
Let us define the subset $\widetilde R_B \subset \C^{B_1}\x\C^{B_2}\x\CP^1$ to be the closure of
$$\{(f_1,f_2,x) \in \C^{B_1} \x \C^{B_2} \x \C^\x \,|\, f_1(x)=f_2(x)=0\}.$$
\end{defin}

The set $\widetilde{R_B}$ is smooth, $R_B$ is its image under the coordinate projection, and they satisfy the following properties:
\begin{lemma}[See Lem.2.7,2.9 and Th.4.12 of \cite{S21}]\label{lemtransvers}Suppose that $\langle B_1, B_2 \rangle=1$.

1) The map $\widetilde R_B \to R_B$ is not local embedding at $(f_1,f_2,x)$ if and only if $x$ is a common root of $(f_1,f_2)$ of multiplicity more than one.

2) Let $x_1$ and $x_2$ be two different common roots of $(f_1,f_2)$ of multiplicity 1. The corresponding local branches of $R_B$ are transversal hypersurfaces, unless $x_1$ and $x_2$ are both from $\C^\x$ and there is $i$ such that both $x_1$ and $x_2$ have multiplicity at least 2 for $f_i$.

3) The set of $(f_1,f_2)$ that have two different common roots $x_1$ and $x_2$ such that both $x_1$ and $x_2$ have multiplicity at least 2 for the same $f_i$, has codimension at least 3.
\end{lemma}

Let $\Sigma$ be the set of $(f,g)$ such that either there is a common root of $(f,g)$ of multiplicity more than one, or there are two common roots of multiplicity at least 2 for the same $f_i$. It has codimension at least 3 and, by Lemma \ref{lemtransvers}, in a small neighborhood of every point $(f_1,f_2)$ of the strata $\tilde S_i:=S_i\setminus\Sigma,\, i=1,2,3,\ldots$, and $\tilde T_j:=T_j\setminus\Sigma,\, j=0,1,2$, the resultant $R_B$ is the union of several smooth pairwise transversal hypersurfaces, whose number is equal to the number of common roots of $(f_1,f_2)$. For the sets $S_i$ (including $S_1$) and $T_0$ these numbers are known from Lemma \ref{lemcompN11}, and for the sets $T_1$ and $T_2$ they are obvious. QED.

\section{Proof of Theorems \ref{th0res} and \ref{th0proj}: Whitney strata of resultants}\label{ssproofs}

\subsection{The Whitney condition (b) and transversal singularity types.} Recall that the transversal singularity type of an algebraic set $R\subset\C^n$ at a smooth point $x$ of its algebraic subset $S$ is the topological type of the singularity $R\cap H$ at $x$, where $H\subset\C^n,\, \dim H=\codim S$, is a germ of a smooth variety, transversally intersecting $S$ at $x$. 
In general the transversal singularity type at a point $x$ is not correctly defined, as the singularity type of $R\cap H$ may depend on the choice of $H$. However, it is a matter of a nice choice of $S\subset R$ to ensure that the transversal singularity type exists and even does not depend on the choice of $x\in S$. 

\begin{exa}\label{exawhitn}
Proposition \ref{propresord} implies that the transversal singularity type of the resultant $R_B$ at a generic point of the strata $\tilde S_i,\, i=1,2,3,\ldots$, and $\tilde T_j,\, j=1,2$ is an ordinary multiple point of order specified in Table 1.
Indeed, denoting any of these strata by $S$, Proposition \ref{propresord}.2 implies that, near any point $x\in S$, the resultant is the union of several pairwise transversal hypersurfaces $H_i\supset S$. 
If a germ of a smooth surface $H$ is transversal to $S$ at $x$, then it is transversal to each $H_i$. Thus $H$ intersects the resultant by the pairwise transversal curves $H\cap H_i$. They by definition form an ordinary multiple point.
\end{exa}

We now want to understand the transversal singularity type of the resultant at the remaining strata $S_0$ and $S_\infty$.
\begin{rem}\label{s0rem}
1. In what follows, with no loss of generality we discuss only $S_0$, because the stratum $S_\infty$ of the resultant $R_{(B_1,B_2)}\subset\C^{B_1}\times\C^{B_2}$ is the stratum $S_0$ of the resultant $R_{(-B_1,-B_2)}\subset\C^{-B_1}\times\C^{-B_2}$ under the natural identification $\C^{B_i}\xrightarrow{\sim}\C^{-B_i}$, sending a polynomial $f(t)$ to $f(t^{-1})$.

2. In what follows, with no loss of generality we assume that $\min B_1=\min B_2=0$, because the resultant $R_{(B_1,B_2)}\subset\C^{B_1}\times\C^{B_2}$ is the resultant $R_{(B_1-\min B_1,B_2-\min B_2)}\subset\C^{B_1-\min B_1}\times\C^{B_2-\min B_2}$ under the natural identification $\C^{B_i}\xrightarrow{\sim}\C^{B_i-\min B_i}$, sending a polynomial $f(t)$ to $f(t)/t^{\min B_i}$.
\end{rem}
If the transversal singularity type of the resultant $R_B$ at a generic point of $S_0$ exists, then it is represented by the $B$-sparse curve singularity, as the following example shows.
\begin{exa}\label{exas0} Under the assumptions of Remark \ref{s0rem} 
$$S_0=\{(f,g)\,|\,f(0)=g(0)=0\}.$$
A point $(f,g)\in S_0$ is contained in an affine plane $H_{f,g}:=\{(f-\alpha,g-\beta)\,|\,\alpha,\beta\in\C\}$ transversal to $S_0$.
The intersection of $H_{f,g}$ with $R_B$ is tautologically the image of the map $\C\to H_{f,g}$ given by $$\alpha=f(t),\,\beta=g(t).\eqno{(*)}$$
The singularity of this image at $(f,g)\in H_{f,g}$ is by definition the $B$-sparse curve singularity, once 
$(f,g)$ 

1) is 0-nondegenerate (see Definition \ref{def0nondeg}), and

2) has no common root in $\C\setminus 0$. 
\end{exa}
In what follows we make use of the notion of Whitney stratification and the Thom-Mather theorem. We recommend e.g. \cite{trotman} as a general reference.
\begin{utver}\label{s0whitn}
At every point of the Zariski open subset $\tilde S_0\subset S_0$ defined by the conditions (1) and (2) of Example \ref{exas0}, and at every point of the strata $\tilde S_i,\, i=1,2,3,\ldots$, and $\tilde T_j,\, j=1,2$ defined by Proposition \ref{propresord}, the resultant satisfies the Whitney condition (b) with respect to $S_0$.
\end{utver}
\begin{proof}
For the strata of Proposition \ref{propresord}, this follows from that proposition by the definition of the Whitney condition (b). So we focus on the stratum $\tilde S_0$ in the setting of Example \ref{exas0}.

All the curves in the family $H_{f,g}\cap R_B$ parameterized by $(f,g)\in S_0$ have the same Milnor number by Theorem \ref{thsparse} and the same multiplicity $\min(B_1\cup B_2\setminus\{0\})$. By a classical result of Zariski (\cite{zariski}, \cite{zariski2}), a family of reduced plane curve germs with a constant Milnor number and multiplicity defines a hypersurface satisfying the Whitney condition (b) with respect to its singular locus.
\end{proof}
\begin{sledst}\label{slwhitn}
The transversal singularity type of the resultant $R_B$ along the stratum $\tilde S_0$ is the $B$-sparse curve singularity. 
\end{sledst}
\begin{proof}
By the Thom-Mather theorem 
(see e.g. \cite{trotman}), Proposition \ref{s0whitn} implies that the transversal singularity type is correctly defined. By Example \ref{exas0}, it is the $B$-sparse curve singularity. \end{proof}

\subsection{Preimages of stratified sets under generic polynomial maps.}

Recall that a smooth map of manifolds $f:M\to N$ is said to be transversal to a submanifold $V\subset N$ at a point $p\in M,\, f(p)\in U$, if the following equivalent conditions are satisfied:

- $f_*T_p M+T_{f(p)}U=T_{f(p)}N$;

- if $U$ is defined by a regular system of equations $g_1=\cdots=g_k=0$ near $f(p)$, then $f^{-1}(U)$ is defined by the regular system of equations $g_1\circ f=\cdots=g_k\circ f=0$ near $p$.

\begin{theor}\label{thtransv}
Let $V\subset \C^n$ be a smooth semialgebraic set (i.e. the difference of two algebraic sets).

1) Let $A_1,\ldots,A_m\subset \Z^n$ be finite sets. Then, for a generic tuple of polynomials $g_i\in\C^{A_i}$, the smooth complete intersection given by the equations $g_1=\ldots=g_m=0$ is transversal to $V\cap\CC^n$ (and in particular their intersection is smooth).

2) Let $B_1,\ldots,B_n\subset\Z^m$ be finite sets. Then, for a generic tuple of polynomials $f_i\in\C^{B_i}$, the map $f_\bullet:\CC^m\to\C^n$ is transversal to $V$ (and in particular the preimage of $V$ is smooth).
\end{theor}
\begin{proof}
Part 1 is well known, see e.g.Theorem 10.7 in \cite{el}. Part 2 for $V\subset\CC^n$ is equivalent to Part 1, applied to the set $\CC^m\times V\subset\CC^m\times\CC^n$ and the graph of the map $f_\bullet$, given as the complete intersection of the hypersurfaces $f_i(x)-y_i=0$. (Here and in what follows, $x=(x_1,\ldots,x_m)$ are the coordinates on $\CC^m$, and $y=(y_1,\ldots,y_n)$ on $\C^n$.)

In the general case, choose arbitrary $I\subset\{1,\ldots,n\}$, and prove that the map $f_\bullet$ is transversal to $V$ at every point of the coordinate torus $$\CC^I=\{y\,|\,y_i=0\Leftrightarrow i\notin I\}.$$
For this, split $V\cap\CC^I$ into smooth strata $V_\alpha$, and note that, according to Part 1, the complete intersection given by the equations
$$ f_i(x)-y_i=0,\,i\in I,\,\mbox{ and }\,f_j(x)=0,\,j\notin I,$$
is transversal to the smooth set
$$\CC^m\times V_\alpha\subset\CC^m\times\CC^I.$$
This implies that the map $f_\bullet$ is transversal to the smooth set $V_\alpha$, and thus to its ambient smooth set $V$ at every point of $V_\alpha$.
\end{proof}

Applying this fact to every stratum of a stratified algebraic set, we have the following.
\begin{sledst}\label{sltransv1} In the setting of Theorem \ref{thtransv}.2, let $W\subset\C^n$ be a union of smooth semialgebraic strata $W_\alpha$.

1) The map $f_\bullet$ is transversal to every stratum, in particular the preimage $U:=f_\bullet^{-1}(W)$ is the union of smooth semialgebraic subsets $U_\alpha:=f_\bullet^{-1}(W_\alpha)$. 

2) If $W_\alpha$ is a Whitney stratification of $W$ (or, more generally, if $W$ has a certain transversal singularity type $\mathcal{S}_\alpha$ along every its stratum $W_\alpha$), then $U_\alpha$ is a Whitney stratification of $U$ (or, respectively, $U$ has the transversal singularity type $\mathcal{S}_\alpha$ along every its stratum $U_\alpha$).
\end{sledst}

If the dimension of $W$ is not too high, this can be stated simpler.

\begin{sledst}\label{sltransv2}
In the setting of Theorem \ref{thtransv}.2, let $W$ be a $p$-dimensional algebraic set with  irreducible components of the singular locus $S_\alpha$.

Let $\mathcal{S}_\alpha$ be the transversal singularity type of $W$ at a generic point of $S_\alpha$. Denote by $q<p$ the dimension of the singular locus, $\max\dim S_\alpha$.

1) The preimage $U:=f^{-1}_\bullet(W)$ has the same codimension $n-p$ as $W$ (with $m<n-p$ meaning $U=\varnothing$).

2) If $m<n-q$, then the image of $f$ does not intersect the singular locus of $W$, and $U$ is smooth.

3) If $m=n-q$, then the image of $f$ transversally intersects the singular locus of $W$ at isolated generic points, and the preimage $U$ has isolated singularities of topological types $\mathcal{S}_\alpha$.
\end{sledst}

\subsection{A connection between singularities of resultants and curve projections}
In order to finish the proof of Theorem \ref{th0res} and prove Theorem \ref{th0proj}, we first establish a connection between their settings by means of Corollary \ref{sltransv2}. 
On one hand, we have the projection $C\subset\CC^2$ of a sparse spatial curve defined by generic polynomial equations $(g_1,g_2)\in\C^{A_1}\oplus\C^{A_2},\,A_i\in\Z^3$. On the other hand, proceeding to the coordinate projections $B_i\subset\Z^1$ of the sets $A_i$, we have the resultant $R_B\subset\C^{B_1}\oplus\C^{B_2}$. How are they connected?

\begin{utver}\label{propresproj}
In this setting, consider the map $$G:\CC^2\to\C^{B_1}\oplus\C^{B_2},\,(y_1,y_2)\mapsto\bigl(g_1(\bullet,y_1,y_2),g_2(\bullet,y_1,y_2)\bigr).\eqno{(*)}$$

1) the map $G$ transversally intersects the singular locus $\sing R_B$ at its generic points;

2) the curve $C$ is the preimage $G^{-1}(R_B)$;

3) the singularities of $C$ are the preimages of $\sing R_B$;

4) the type of every such singular point $y\in C$ is the transversal singularity type of $\sing R_B$ at its point $G(y)$.
\end{utver}
This is a special case of Corollary \ref{sltransv2}. We shall use it both ways to finish the proof of the main theorems.

\vspace{1ex}

{\it Proof of Theorem \ref{th0proj}.}
1. By Proposition \ref{propresproj}.1\&4, the curve $C$ has singularities exactly at the preimages of $\sing R_B$ under the map $(*)$, and the type of each singularity coincides with the transversal singularity type of the resultant at a generic point of the corresponding component of the singular locus. These singularity types are described by Example \ref{exawhitn} and Corollary \ref{slwhitn} (and we use the notation therefrom to finish the proof).

2. Since the preimages of the singularity strata $\tilde S_i,\, i\ne 1$, and $\tilde T_j,\, j=0,1,2$ of $\sing R_B$ are given by the generic systems of equations $(1-4)$ of Theorem \ref{th0proj}, the number of points in each of these preimages (i.e. the number of singular points of  the corresponding type in the curve $C$) can be found from the BKK formula \cite{bernst}. The answer is the mixed volume of the Newton polytopes of the equations, specified in the statement of Theorem \ref{th0proj}.
\hfill$\square$

\vspace{1ex}

{\it Proof of Theorem \ref{th0res}.} All statements except for the degrees of the strata in Table 1 are covered by Propositions \ref{propresord} and \ref{propdoublestrara}, Example \ref{exawhitn} and Corollary \ref{slwhitn}. In order to compute the degrees of the strata, in the setting of Theorem \ref{th0proj}, choose $$A_i:=B_i\times(\mbox{the standard triangle in }\Z^2).$$ For this choice, the image of the map $(*)$ of Proposition \ref{propresproj} is a generic codimension 2 plane $P$ in $\C^{B_1}\oplus\C^{B_2}$. The sought degree of each stratum equals the number of intersections  of this stratum with $P$, i.e. the number of singularities of the corresponding type in the curve $C$. The number of singularities can be computed from Theorem \ref{th0proj} and Observation \ref{obnumdouble}. It is directly verified that, for $A_i$ as above, the numbers of singularities coincide with the numbers in the second column of Table 1. \hfill$\square$

\end{document}